\documentclass[10pt]{article}

\usepackage[T1]{fontenc}
\usepackage{lmodern}

\usepackage{amsmath, amssymb, amsthm}
\usepackage[margin=2.5cm]{geometry}
\usepackage{color, graphicx, float}
\usepackage{subcaption}
\usepackage{caption}
\usepackage{changepage}
\usepackage[breakable]{tcolorbox}
\usepackage{multirow}
\usepackage{rotating}

\usepackage{etoolbox}
\makeatletter
\patchcmd{\@maketitle}{\LARGE \@title}{\LARGE\bfseries\@title}{}{}

\renewcommand{\@seccntformat}[1]{\csname the#1\endcsname.\quad}
\makeatother

\definecolor{darkblue}{rgb}{0,0,.5}

\usepackage{hyperref}
\hypersetup{
	colorlinks=true,		
	linkcolor=darkblue,		
	citecolor=darkblue,		
	urlcolor=darkblue 		
}

\makeatletter
\def\th@plain{%
	\thm@notefont{}
	\itshape 
}
\def\th@definition{%
	\thm@notefont{}
	\normalfont 
}

\renewenvironment{proof}[1][\proofname]{\par
	\normalfont
	\topsep0\p@\@plus3\p@ \trivlist
	\item[\hskip\labelsep\itshape
	#1\@addpunct{.}]\ignorespaces
}{%
	\qed\endtrivlist
}
\makeatother

\newtheorem{theorem}{Theorem}[section]
\newtheorem{lemma}[theorem]{Lemma}

\theoremstyle{definition}

\theoremstyle{definition}

\theoremstyle{definition}
\newtheorem{remark}[theorem]{Remark}
\theoremstyle{definition}
\newtheorem{algorithm}[theorem]{Algorithm}
\theoremstyle{definition}

\usepackage[shortlabels]{enumitem}

\newcommand{\dom}{\ensuremath{\operatorname{dom}}}
\newcommand{\prox}{\ensuremath{\operatorname{prox}}}
\newcommand{\argmin}{\ensuremath{\operatorname*{argmin}}}
\newcommand{\epi}{\ensuremath{\operatorname{epi}}}
\newcommand{\dist}{\ensuremath{\operatorname{dist}}}

\allowdisplaybreaks

\newcounter{step}
\newcommand\step[1]{%
	\refstepcounter{step}	
	\vskip 0.25\baselineskip
	\ifx\hfuzz#1\hfuzz
		\item[~\(\triangleright\)~\textbf{Step~\arabic{step}.}]
	\else
		\item[~\(\triangleright\)~\textbf{Step~\arabic{step}}] (\texttt{#1})\textbf{.}%
	\fi
}

\begin{document}

\title{A proximal subgradient algorithm with extrapolation for structured nonconvex nonsmooth problems}

\author{\small
Tan Nhat Pham\thanks{Centre for Smart Analytics and Centre for New Energy Transition Research, Federation University Australia, Ballarat, VIC 3353, Australia.
E-mail: \texttt{tanp@students.federation.edu.au}.},
~
Minh N. Dao\thanks{School of Science, RMIT University, Melbourne, VIC 3000, Australia.
E-mail: \texttt{minh.dao@rmit.edu.au}.},
~
Rakibuzzaman Shah\thanks{Centre for New Energy Transition Research, Federation University Australia, Ballarat, VIC 3353, Australia. 
E-mail: \texttt{m.shah@federation.edu.au}.},
~
Nargiz Sultanova\thanks{Centre for Smart Analytics, Federation University Australia, Ballarat, VIC 3353, Australia.
E-mail: \texttt{n.sultanova@federation.edu.au}.},
~
Guoyin Li\thanks{Department of Applied Mathematics, University of New South Wales, Sydney 2052, Australia.
E-mail: \texttt{g.li@unsw.edu.au}.},
~~and~
Syed Islam\thanks{Centre for New Energy Transition Research, Federation University Australia, Ballarat, VIC 3353, Australia. 
E-mail: \texttt{s.islam@federation.edu.au}.}
}

\date{\today}

\maketitle

\begin{abstract}
In this paper, we consider a class of structured nonconvex nonsmooth optimization problems, in which the objective function is formed by the sum of a possibly nonsmooth nonconvex function and a differentiable function with Lipschitz continuous gradient, subtracted by a weakly convex function. This general framework allows us to tackle problems involving nonconvex loss functions and problems with specific nonconvex constraints, and it has many applications such as signal recovery, compressed sensing, and optimal power flow distribution. We develop a proximal subgradient algorithm with extrapolation for solving these problems with guaranteed subsequential convergence to a stationary point. The convergence of the whole sequence generated by our algorithm is also established under the widely used Kurdyka--{\L}ojasiewicz property. To illustrate the promising numerical performance of the proposed algorithm, we conduct numerical experiments on two important nonconvex models. These include a compressed sensing problem with a nonconvex regularization and an optimal power flow problem with distributed energy resources.
\end{abstract}


\smallskip
\noindent{\bf Mathematics Subject Classification (MSC 2020):}
90C26,	
49M27,	
65K05.	

\section{Introduction}

In this work, we consider the structured optimization problem
\begin{equation}\label{eq:P}
\min_{x\in C} F(x) :=f(x)+h(Ax)-g(x), 
\tag{P}
\end{equation}
where $C$ is a nonempty closed subset of a finite-dimensional real Hilbert space $\mathcal{H}$, $A$ is a linear mapping from $\mathcal{H}$ to another finite-dimensional real Hilbert space $\mathcal{K}$, $f\colon \mathcal{H}\to (-\infty, +\infty]$ is a proper lower semicontinuous (possibly nonsmooth and nonconvex) function, $h\colon \mathcal{K}\to \mathbb{R}$ is a differentiable (possibly nonconvex) function whose gradient is Lipschitz continuous with modulus $\ell$, and $g\colon \mathcal{H}\to (-\infty, +\infty]$ is a continuous weakly convex function with modulus $\beta$ on an open convex set containing $C$. This broad optimization problem has many important applications in diverse areas, including power control problems \cite{Cheng2012}, compressed sensing \cite{Lou2017}, portfolio optimization, supply chain problem, image segmentation, and others \cite{LeThi2018}. 

In particular, the model problem \eqref{eq:P} covers two of the most general models in the literature. Firstly, in statistical learning, the following optimization model is often used
\begin{align}\label{eq:stats}
\min _{x\in\mathbb{R}^d} \left(\varphi(x)+\gamma \, r(x)\right),
\end{align}
where $\varphi$ is called a loss function which measures the data misfitting, $r$ is a regularization which promotes specific structure in the solution such as sparsity, and $\gamma>0$ is a weighting parameter. Typical choices of the loss function are the least square loss function $\varphi(x)=\frac{1}{2}\|Ax-b\|^2$ where $A\in \mathbb{R}^{m\times d}$ and $b \in \mathbb{R}^m$ and the logistic loss function, which are both convex. In the literature, nonconvex loss functions have also received increased attentions. Some popular nonconvex loss functions include the ramp loss function \cite{Ramp_loss2022,Ramp_loss2017} and the Lorentzian norm \cite{Lorentian}. In addition, \cite{Ahn_2017} recently showed that many regularization $r$ used in the literature can be written as difference of two convex functions, and so, the model \eqref{eq:stats} can be formulated into problem \eqref{eq:P}. These include popular regularizations such as the smoothly clipped absolute deviation (SCAD) \cite{Antoniadis_1997}, the indicator function of cardinality constraint \cite{Gotoh_2017}, $L_1-L_2$ regularization \cite{Lou2017}, or minimax concave penalty (MCP) \cite{Zhang_2010}. Therefore, problem \eqref{eq:P} can be interpreted as a problem with the form \eqref{eq:stats} whose objective function is the sum of a nonconvex and nonsmooth loss function and a regularization which can be expressed as a specific form of difference-of-(possibly) nonconvex functions\footnote{Indeed, note that any smooth function with Lipschitz gradient function  is weakly convex. By adding and subtracting $\alpha\|x\|^2$ for large $\alpha>0$, our model problem (P) can also be mathematically reduced to the form \eqref{eq:stats} whose objective function is the sum of a nonconvex and nonsmooth loss function and a  difference-of-convex regularization.}.
Secondly, in the case when $C =\mathbb{R}^d$ and $A$ is the identity mapping, problem \eqref{eq:P} reduces to
\begin{equation}\label{eq:P'}
\min_{x\in \mathbb{R}^d} \left(f(x)+h(x)-g(x)\right), 
\end{equation}
referred as the general difference-of-convex (DC) program, which is a broad class of optimization problems studied in the literature. To solve problem \eqref{eq:P'} under the convexity of $g$, a \emph{generalized proximal point algorithm} was developed in \cite{An2016}. For the case when both $f$ and $g$ are convex, \cite{Phan2018} provided an \emph{accelerated difference-of-convex algorithm} incorporating Nesterov's acceleration technique into the standard \emph{difference-of-convex algorithm} (DCA) to improve its performance, while \cite{Liu2022} proposed an \emph{inexact successive quadratic approximation method}. When $f$, $h$, and $g$ are all required to be convex, a \emph{proximal difference-of-convex algorithm with extrapolation} (pDCAe) was proposed in \cite{Wen2017}, and there are also other existing studies (e.g., \cite{Lu2019,Lu2018}) that developed algorithms to solve such a problem.

In the cases where the loss function $f$ is smooth and the regularization $r$ is prox-friendly in the sense that its proximal operator can be computed efficiently, the proximal gradient method is a widely used algorithm for solving \eqref{eq:stats} (for example, see \cite{Beck2017}). 
Moreover, incorporating information from previous iterations to accelerate the proximal algorithm while trying not to significantly increase the computational cost has also been a  research area which receives a lot of attention. One such approach is extrapolation technique. In this approach, \emph{momentum} terms that involve the information from previous iterations are used to update the current iteration. Such techniques have been successfully implemented and achieved significant results, including Polyak's heavy ball method \cite{Polyak1964}, Nesterov's techniques \cite{Nesterov2021,Nesterov2018}, and the fast iterative shrinking-threshold algorithm (FISTA) \cite{Beck2009}. In particular, extrapolation techniques have shown competitive results for optimization problems that involve sum of convex functions \cite{Attouch2018}, difference of convex functions \cite{Wen2017,Lu2018}, and ratio of nonconvex and nonsmooth functions \cite{Bo2021}.

 In view of these successes, this paper proposes an extrapolated proximal subgradient algorithm for solving problem \eqref{eq:P}. In our work, comparing to the literature, the convexity and smoothness of the loss functions $f$ are relaxed. We also allow a closed feasible set $C$ instead of optimizing over the whole space. This general framework allows us to tackle problems involving nonconvex loss functions such as Lorentzian norm and problems with specific nonconvex constraints such as spherical constraint. We then prove that the sequence generated by the algorithm is bounded and any of its cluster points is a stationary point of the problem. We also prove the convergence of the full sequence under the assumption of Kurdyka--{\L}ojasiewicz property. We then evaluate the performance of the proposed algorithm on a compressed sensing problem for both convex and nonconvex loss functions together with the recently proposed nonconvex $L_1-L_2$ regularization. Finally, we formulate an optimal power flow problem considering photovoltaic systems placement, and address it using our algorithm.

The rest of this paper is organized as follows. Section~\ref{sec:preliminaries} provides preliminary materials used in this work. In Section~\ref{sec:algorithm}, we introduce our algorithm with guaranteed subsequential convergence and full sequential convergence. Section~\ref{sec:casestudy} presents the numerical experiments, and conclusion is given in Section~\ref{sec:conclusion}.

\section{Premilinaries}
\label{sec:preliminaries}

Throughout this paper, $\mathcal{H}$ is a finite-dimensional real Hilbert space with inner product $\langle \cdot, \cdot \rangle$ and the induced norm $\|\cdot\|$. We use the notation $\mathbb{N}$ for the set of nonnegative integers, $\mathbb{R}$ for the set of real numbers, $\mathbb{R}_+$ for the set of nonnegative real numbers, and $\mathbb{R}_{++}$ for the set of the positive real numbers.

Let $f\colon \mathcal{H}\to \left[-\infty,+\infty\right]$. The \emph{domain} of $f$ is $\dom f :=\{x\in \mathcal{H}: f(x) <+\infty\}$ and the \emph{epigraph} of $f$ is $\epi f := \{(x,\rho)\in \mathcal{H}\times \mathbb{R}: f(x)\leq \rho\}$. The function $f$ is \emph{proper} if $\dom f \neq \varnothing$ and it never takes the value $-\infty$, \emph{lower semicontinuous} if its epigraph is a closed set, and \emph{convex} if its epigraph is a convex set. We say that $f$ is \emph{weakly convex} if $f+\frac{\alpha}{2}\|\cdot\|^2$ is convex for some $\alpha\in \mathbb{R}_+$. The \emph{modulus} of the weak convexity is the smallest constant $\alpha$ such that $f+\frac{\alpha}{2}\|\cdot\|^2$ is convex. Given a subset $C$ of $\mathcal{H}$, the \emph{indicator function} $\iota_C$ of $C$ is defined by $\iota_C(x) :=0$ if $x\in C$, and $\iota_C(x) :=+\infty$ if $x\notin C$. If $f+\iota_C$ is weakly convex with modulus $\alpha$, then $f$ is said to be \emph{weakly convex on $C$ with modulus $\alpha$}. Some examples of weakly convex functions are quadratic functions, convex functions, and differentiable functions with Lipschitz continuous gradient.

Let $x\in \mathcal{H}$ with $\lvert f(x) \rvert <+\infty$. The \emph{Fr\'echet subdifferential} of $f$ at $x$ is defined by

\begin{equation*}
\widehat{\partial} f(x) :=\left\{x^*\in \mathcal{H}:\; \liminf_{y\to x}\frac{f(y)-f(x)-\langle x^*,y-x\rangle}{\|y-x\|}\geq 0\right\}
\end{equation*}
and the \emph{limiting subdifferential} of $f$ at $x$ is defined by

\begin{equation*}
\partial_L f(x) :=\left\{x^*\in \mathcal{H}:\; \exists x_n\stackrel{f}{\to}x,\; x_n^*\to x^* \text{~~with~~} x_n^*\in \widehat{\partial} f(x_n)\right\},
\end{equation*}
where the notation $y\stackrel{f}{\to}x$ means $y\to x$ with $f(y)\to f(x)$. In the case where $\lvert f(x) \rvert =+\infty$, both Fr\'echet subdifferential and limiting subdifferential of $f$ at $x$ are defined to be the empty set. The \emph{domain} of $\partial_L f$ is given by $\dom \partial_L f :=\{x\in \mathcal{H}: \partial_L f(x) \neq \varnothing\}$. It can be directly verified from the definition that the limiting subdifferential has the \emph{robustness property}

\begin{equation*}
\partial_L f(x) =\left\{x^*\in \mathcal{H}:\; \exists x_n \stackrel{f}{\to}x,\; x_n^*\to x^* \text{~~with~~} x_n^*\in \partial_L f(x_n)\right\}.
\end{equation*}

Next, we revisit some important properties of the limiting subdifferential.

\begin{lemma}[Sum rule]
\label{lemma:sum}
Let $x \in \mathcal{H}$ and let $f, g\colon \mathcal{H} \to (-\infty, +\infty]$ be proper lower semicontinuous functions. Suppose that $f$ is finite at $x$ and $g$ is locally Lipschitz around ${x}$. Then $\partial_L (f + g)({x}) \subseteq \partial_L f({x})+\partial_L g({x})$. Moreover, if $g$ is strictly differentiable at ${x}$, then $\partial_L(f + g)({x}) = \partial_Lf({x}) + \nabla g({x})$.
\end{lemma}
\begin{proof}
This follows from \cite[Proposition 1.107(ii) and Theorem 3.36]{Mordukhovich2006}.
\end{proof}

The following result, whose proof is included for completeness, is similar to \cite[Lemma~2.9]{Dao2019}.
\begin{lemma}[Upper semicontinuity of subdifferential]
\label{l:upsemicont}
Let $f\colon \mathcal{H} \to [-\infty, +\infty]$ be Lipschitz continuous around $x \in \mathcal{H}$, let $(x_n)_{n\in \mathbb{N}}$ be a sequence in $\mathcal{H}$ converging to $x$, and let, for each ${n\in \mathbb{N}}$, $x_n^* \in \partial_L f(x_n)$. Then $(x_n^*)_{n\in \mathbb{N}}$ is bounded with all cluster points contained in $\partial_L f(x)$.
\end{lemma}
\begin{proof}
By the Lipschitz continuity of $f$ around $x$, there are a neighborhood $V$ of $x$ and a constant $\ell_V \in \mathbb{R}_+$ such that $f$ is Lipschitz continuous on $V$ with modulus $\ell_V$. Then, by \cite[Corollary~1.81]{Mordukhovich2006}, for all $v\in V$ and $v^*\in \partial_L f(v)$, one has $\|v^*\|\leq \ell_V$. Since $x_n \to x$ as $n\to +\infty$, there is $n_0\in \mathbb{N}$ such that, for all $n\geq n_0$, $x_n \in V$, which implies that $\|x^*_n\|\leq \ell_V$. This means $(x_n^*)_{n\in \mathbb{N}}$ is bounded.  

Now, let $x^*$ be a cluster point of $(x_n^*)_{n\in \mathbb{N}}$, i.e., there exists a subsequence $(x_{k_n}^*)_{n\in \mathbb{N}}$ such that $x_{k_n}^*\to x^*$ as $n\to +\infty$. On the other hand, we have from the convergence of $(x_n)_{n\in \mathbb{N}}$ and the Lipschitz continuity of $f$ around $x$ that $x_{k_n} \stackrel{f}{\to} x$. Therefore, $x^*\in \partial_L f(x)$ due to the robustness property of the limiting subdifferential.
\end{proof}

We end this section with the definitions of stationary points for the problem \eqref{eq:P}. A point $\overline{x}\in C$ is said to be a 
\begin{itemize}
\item 
\emph{stationary point} of \eqref{eq:P} if $0\in \partial_L(f+\iota_C+h\circ A-g)(\overline{x})$,
\item  
\emph{lifted stationary point} of \eqref{eq:P} if $0\in \partial_L(f+\iota_C)(\overline{x})+A^*\nabla h(A\overline{x}) -\partial_L g(\overline{x})$. 
\end{itemize}
Here $A^*$ is the adjoint mapping of the linear mapping $A$.

\section{Proximal subgradient algorithm with extrapolation}
\label{sec:algorithm} 

We now propose our extrapolated proximal subgradient algorithm for solving problem \eqref{eq:P} with guaranteed convergence to stationary points.

\begin{tcolorbox}[
	left=0pt,right=0pt,top=0pt,bottom=0pt,
	colback=blue!10!white, colframe=blue!50!white,
  	boxrule=0.2pt,
  	breakable]
\begin{algorithm}[Proximal subgradient algorithm with extrapolation]
\label{algo:extrapolated}
\step{}
Let $x_{-1} =x_0\in C$ and set $n =0$. Let $\overline{\lambda} \in \mathbb{R}_+$, $\overline{\mu} \in \mathbb{R}_+$, and $\delta \in \mathbb{R}_{++}$.

\step{}\label{step:main}
Let $g_n\in \partial_L g(x_n)$, $u_n=x_n+\lambda_n(x_n-x_{n-1})$ with $\lambda_n \in [0, \overline{\lambda}]$, and $v_n=x_n+\mu_n(x_n-x_{n-1})$ with $\mu_n \in [0, \overline{\mu}\tau_n]$. Choose $\tau_n\in \left(0, 1/(\beta +2\delta +\ell\|A\|^2\left(2\overline{\lambda}+1\right) +2\overline{\mu}) \right]$ and compute
\begin{align*}
x_{n+1} \in 
\argmin_{x\in C} \left( f(x) +\tfrac{1}{2\tau_n}\|x-v_n+\tau_n A^*\nabla h(Au_n)-\tau_n g_n\|^2 \right).
\end{align*}
\step{}
If a termination criterion is not met, set $n =n+1$ and go to Step~\ref{step:main}.
\end{algorithm}
\end{tcolorbox}

\begin{remark}[Discussion of the algorithm structure and extrapolation parameters]
\label{remark:algo}
Some comments on Algorithm~\ref{algo:extrapolated} are in order.
\begin{enumerate}
\item 
Recalling that the proximal operator of a proper function $\phi\colon \mathcal{H}\to \left(-\infty,+\infty\right]$ is defined by
\begin{align*}
\prox_{\phi}(x) =\argmin_{y\in \mathcal{H}}\left(\phi(y) + \frac{1}{2}\|y-x\|^2\right),
\end{align*}
we see that the update of $x_{n+1}$ in Step~2 can be written as
\begin{align*}
x_{n+1} \in \prox_{\tau_n (f+\iota_C)}(v_n-\tau_n A^*\nabla h(Au_n)+\tau_n g_n).
\end{align*}
This can be done efficiently for various specific structures of $f$ and $C$. For example, when $f$ is a convex quadratic function and $C$ is a polyhedral set, computing the proximal operator of $\tau_n (f+\iota_C)$ is equivalent to solving a convex quadratic programming problem. When $f$ is a nonconvex quadratic function and $C$ is the unit sphere, this reduces to a trust region problem which can be solved as a generalized eigenvalue problem or a semi-definite programming problem. In addition, the proximal operator can also have closed form solution for some nonconvex and nonsmooth functions, e.g., $f(x) =\|x\|_1 - \alpha\|x\|$ with $\alpha\in \mathbb{R}_+$ (see \cite[Lemma 1]{Lou2017}). For further tractable cases, see, e.g., \cite[Remark~4.1]{Bo2021}. 

\item 
Let us consider the case when $A$ is the identity mapping and $C = \mathcal{H}$. We fix an arbitrary $\tau \in \left(0, 1/\ell\right)$ and choose $\overline{\lambda} =\overline{\mu} =0$ (which yields $\lambda_n = \mu_n =0$), $\delta \in \left(0, 1/(2\tau) -\ell/2\right)$, and $\tau_n =\tau$. Then the update of $x_{n+1}$ in Step~2 becomes
\begin{align*}
x_{n+1}\in \prox_{\tau f}(x_n-\tau \nabla h(x_n) + \tau g_n),
\end{align*}
which is the so-called \emph{generalized proximal point algorithm} (GPPA) in \cite{An2016}, where $g$ is assumed to be convex (In this case, $\beta =0$ and $1/\tau_n = 1/\tau >2\delta +\ell =\beta +2\delta +\ell\|A\|^2\left(2\overline{\lambda}+1\right) +2\overline{\mu}$). 

\item
In the case where $h\equiv0$, the objective function $F$ reduces to $f-g$ and the update of $x_{n+1}$ in Step~2 reduces to
\begin{align*}
x_{n+1} \in \prox_{\tau_n (f+\iota_C)}(v_n+\tau_n g_n).
\end{align*}
In turn, if $C = \mathcal{H}$ and $\mu_n=0$, Algorithm~\ref{algo:extrapolated} reduces to the \emph{proximal linearized algorithm} proposed in \cite{Souza2015} which requires that $f$ and $g$ are convex. 

\item 
When $g \equiv 0$, $A$ is the identity mapping, and $C = \mathcal{H}$, the objective function reduces to $f+h$. By choosing $\lambda_n =\mu_n$, we have
\begin{align*}
    x_{n+1} \in \prox_{\tau_n f}(u_n-\tau_n\nabla h(u_n))
\end{align*}
and Algorithm~\ref{algo:extrapolated} reduces to the \emph{inertial forward-backward algorithm} studied in \cite{Attouch2018}, in which an additional requirement of the convexity of $h$ is imposed.

\item 
Motivated by the popular parameter used in FISTA and also its variants \cite[Chapter~10]{Beck2017}, a plausible option for extrapolation parameters $\lambda_n$ and $\mu_n$ (which will be used in our computation later) is that
\begin{equation*}
\lambda_n =\overline{\lambda}\frac{\kappa_{n-1}-1}{\kappa_n} \text{~~and~~} \mu_n =\overline{\mu}\tau_n\frac{\kappa_{n-1}-1}{\kappa_n},    
\end{equation*}
where $\kappa_{-1} =\kappa_0 =1$ and $\kappa_{n+1} =\frac{1+\sqrt{1+4\kappa_n^2}}{2}$. It can be seen that, for all $n\in \mathbb{N}$, $1\leq \kappa_{n-1} <\kappa_n +1$, and so $\lambda_n \in [0, \overline{\lambda}]$ and $\mu_n \in [0, \overline{\mu}\tau_n]$. We can also reset $\kappa_{n-1} =\kappa_n =1$ whenever $n$ is a multiple of some fix integer $n_0$.   
\end{enumerate}
\end{remark}

From now on, let $(x_n)_{n\in \mathbb{N}}$ be a sequence generated by Algorithm~\ref{algo:extrapolated}. Under suitable assumptions, we show in the next theorem that $(x_n)_{n\in \mathbb{N}}$ is bounded and any of its cluster points is a stationary point of problem \eqref{eq:P}.

\begin{theorem}[Subsequential convergence]
\label{theorem:cvg_ex}
For problem \eqref{eq:P}, suppose that the function $F$ is bounded from below on $C$ and that the set $C_0 :=\{x\in C: F(x) \leq F(x_0)\}$ is bounded. Set $c :=\frac{1}{2}(\ell\|A\|^2\overline{\lambda}+\overline{\mu})$. Then the following statements hold:
\begin{enumerate}
\item\label{t:cvg_decrease2}
For all $n\in \mathbb{N}$,
\begin{align}\label{eq:decrease}
&(F(x_{n+1}) + c\|x_{n+1}-x_{n}\|^2)+\delta\|x_{n+1}-x_{n}\|^2 
\leq  F(x_n)+ c\|x_{n}-x_{n-1}\|^2    
\end{align}
and the sequence $(F(x_n))_{n\in \mathbb{N}}$ is convergent.

\item\label{t:cvg_seq2}
The sequence $(x_n)_{n\in \mathbb{N}}$ is bounded and $x_{n+1}-x_n\to 0$ as $n\to +\infty$.
\item\label{t:cvg_crit2}
Suppose that $\liminf_{n\to +\infty} \tau_n =\overline{\tau} >0$ and let $\overline{x}$ be a cluster point of $(x_n)_{n\in \mathbb{N}}$. Then $\overline{x}\in C\cap \dom f$, $F(x_n)\to F(\overline{x})$, and $\overline{x}$ is a lifted stationary point of \eqref{eq:P}. Moreover, $\overline{x}$ is a stationary point of \eqref{eq:P} provided that $g$ is strictly differentiable on an open set containing $C\cap \dom f$.
\end{enumerate}
\end{theorem}
\begin{proof}
\ref{t:cvg_decrease2} \& \ref{t:cvg_seq2}: 
We see from Step~2 of Algorithm~\ref{algo:extrapolated} that, for all $n\in \mathbb{N}$, $x_n\in C$ and
\begin{align*}
x_{n+1} &\in 
\argmin_{x\in C} \left( f(x) +\frac{1}{2\tau_n}\|x-v_n+\tau_n A^*\nabla h(Au_n)-\tau_n g_n\|^2 \right)\\
&=\argmin_{x\in C} \left( f(x) +\frac{1}{2\tau_n}\|x-v_n\|^2 +\langle A^*\nabla h(Au_n)-g_n, x-v_n\rangle \right)\\
&=\argmin_{x\in C} \left( f(x) +\langle A^*\nabla h(Au_n), x-u_n\rangle -\langle g_n, x-v_n\rangle  \right.\\
&\qquad \left. +\langle A^*\nabla h(Au_n), u_n- v_n\rangle +\frac{1}{2\tau_n}\|x-v_n\|^2 \right)\\
&= \argmin_{x\in C} \left( f(x) +\langle \nabla h(Au_n), Ax-Au_n\rangle -\langle g_n, x-v_n\rangle+\frac{1}{2\tau_n}\|x-v_n\|^2\right).
\end{align*}
Therefore, for all $n\in \mathbb{N}$ and all $x\in C$,
\begin{align*}
&f(x) +\langle \nabla h(Au_n), Ax-Au_n\rangle -\langle g_n, x-v_n\rangle +\frac{1}{2\tau_n}\|x-v_n\|^2\\ 
&\geq f(x_{n+1}) +\langle \nabla h(Au_n), Ax_{n+1}-Au_n\rangle -\langle g_n, x_{n+1}-v_n\rangle+\frac{1}{2\tau_n}\|x_{n+1}-v_n\|^2,    
\end{align*}
or equivalently,
\begin{align}\label{eq:fx+}
f(x) &\geq f(x_{n+1}) +\langle \nabla h(Au_n), Ax_{n+1}-Ax\rangle  -\langle g_n, x_{n+1}-x\rangle \notag \\
&\qquad +\frac{1}{2\tau_n}(\|x_{n+1}-v_n\|^2 -\|x-v_n\|^2).
\end{align}
By the Lipschitz continuity of $\nabla h$, we derive from \cite[Lemma~1.2.3]{Nesterov2018} that
\begin{align*}
&\langle \nabla h(Au_n), Ax_{n+1}-Ax_n\rangle \\
&= \langle \nabla h(Ax_n), Ax_{n+1}-Ax_n\rangle+\langle \nabla h(Au_n) -\nabla h(Ax_n), Ax_{n+1}-Ax_n\rangle \\
&\geq h(Ax_{n+1}) -h(Ax_n) -\frac{\ell}{2}\|Ax_{n+1}-Ax_n\|^2 \\
&\qquad -\|\nabla h(Au_n) -\nabla h(Ax_n)\| \|Ax_{n+1}-Ax_n\| \\
&\geq h(Ax_{n+1}) -h(Ax_n) -\frac{\ell\|A\|^2}{2}\|x_{n+1}-x_n\|^2 -\ell\|A\|^2\|u_n-x_n\| \|x_{n+1}-x_n\|.
\end{align*}
As $g_n\in \partial_L g(x_n)$ and $x_n,~ x_{n+1}\in C$, it follows from the weak convexity of $g$ and \cite[Lemma~4.1]{Bo2021} that 
\begin{equation*}
\langle g_n, x_{n+1}-x_n\rangle \leq g(x_{n+1}) -g(x_n) +\frac{\beta}{2}\|x_{n+1}-x_n\|^2.   
\end{equation*}
Letting $x =x_n\in C$ in \eqref{eq:fx+} and combining with the last two inequalities, we obtain that
\begin{align*}
&f(x_n) +h(Ax_n) -g(x_n) \\
&\geq f(x_{n+1}) +h(Ax_{n+1}) -g(x_{n+1})-\left(\frac{\ell\|A\|^2}{2}+\frac{\beta}{2}\right)\|x_{n+1}-x_n\|^2 \\
&\qquad  -\ell\|A\|^2\|u_n-x_n\| \|x_{n+1}-x_n\|+\frac{1}{2\tau_n}(\|x_{n+1}-v_n\|^2 -\|x_n-v_n\|^2).
\end{align*}
By the definition of $u_n$ and $v_n$, we have 
$x_n-u_n =-\lambda_n(x_n-x_{n-1})$, $x_{n+1}-v_n =(x_{n+1}-x_n) -\mu_n(x_n-x_{n-1})$, $x_n-v_n =-\mu_n(x_n-x_{n-1})$, and so
\begin{align*}
F(x_n) 
&\geq F(x_{n+1}) -\left(\frac{\ell\|A\|^2}{2}+\frac{\beta}{2}\right)\|x_{n+1}-x_n\|^2 -\ell\|A\|^2\lambda_n\|x_n-x_{n-1}\| \|x_{n+1}-x_n\|\\
&\qquad +\frac{1}{2\tau_n}(\|x_{n+1}-x_n\|^2-2\mu_n\langle x_{n+1}-x_n, x_n-x_{n-1} \rangle)\\
&\geq F(x_{n+1}) +\left(\frac{1}{2\tau_n}-\frac{\ell\|A\|^2}{2}-\frac{\beta}{2}\right)\|x_{n+1}-x_n\|^2 \\
&\qquad -\left(\ell\|A\|^2\lambda_n+\frac{\mu_n}{\tau_n}\right)\|x_{n+1}-x_n\| \|x_n-x_{n-1}\| \\
&\geq F(x_{n+1})-\left(\frac{\ell\|A\|^2\lambda_n}{2}+\frac{\mu_n}{2\tau_n}\right)\|x_n-x_{n-1}\|^2  \\
&\qquad +\left(\frac{1}{2\tau_n}-\frac{\ell\|A\|^2}{2}-\frac{\beta}{2}-\frac{\ell\|A\|^2\lambda_n}{2}-\frac{\mu_n}{2\tau_n}\right)\|x_{n+1}-x_n\|^2,
\end{align*}
where we have used $\langle x_{n+1}-x_n, x_n-x_{n-1} \rangle \leq \|x_{n+1}-x_n\| \|x_n-x_{n-1}\| \leq  \frac{1}{2}(\|x_{n+1}-x_n\|^2 + \|x_n-x_{n-1} \|^2)$. Rearranging terms yields
\begin{align*}
&F(x_n) +\left(\frac{\ell\|A\|^2\lambda_n}{2}+\frac{\mu_n}{2\tau_n}\right)\|x_n-x_{n-1}\|^2 \\ 
&\geq F(x_{n+1}) +\left(\frac{1}{2\tau_n}-\frac{\ell\|A\|^2}{2}-\frac{\beta}{2}-\frac{\ell\|A\|^2\lambda_n}{2}-\frac{\mu_n}{2\tau_n}\right)\|x_{n+1}-x_n\|^2.
\end{align*}
Since $\lambda_n \in [0, \overline{\lambda}]$, $\mu_n \in [0, \overline{\mu}\tau_n]$, and 
$1/\tau_n \geq \beta +2\delta +\ell\|A\|^2\left(2\overline{\lambda}+1\right) +2\overline{\mu}$, it follows that
\begin{align*}
F(x_n) +\frac{1}{2}(\ell\|A\|^2\overline{\lambda}+\overline{\mu})\|x_n-x_{n-1}\|^2 \geq F(x_{n+1})+\frac{1}{2}(2\delta + \ell\|A\|^2\overline{\lambda}+\overline{\mu})\|x_{n+1}-x_n\|^2,
\end{align*}
which proves \eqref{eq:decrease}.

Recalling $c =\frac{1}{2}(\ell\|A\|^2\overline{\lambda}+\overline{\mu})$ and setting $\mathcal{F}_n :=F(x_n) +c\|x_n-x_{n-1}\|^2$, we have
\begin{equation}\label{eq:30}
    \mathcal{F}_n \geq \mathcal{F}_{n+1} + \delta\|x_{n+1}-x_n\|^2.
\end{equation}
Since $\delta >0$, the sequence $(\mathcal{F}_n)_{n\in \mathbb{N}}$ is nonincreasing. Since $F$ is bounded below on $C$, the sequence $(\mathcal{F}_n)_{n\in \mathbb{N}}$ is bounded below, and it is therefore convergent.
After rearranging \eqref{eq:30} and performing telescoping, we obtain that, for all $m\in \mathbb{N}$,
\begin{equation*}
\delta\sum_{n=0}^m \|x_{n+1}-x_{n}\|^2 \leq\sum_{n=0}^{m} (\mathcal{F}_n- \mathcal{F}_{n+1})=\mathcal{F}_0-\mathcal{F}_{m+1}.
\end{equation*}
Denoting $\overline{\mathcal{F}} :=\lim_{n\to +\infty} \mathcal{F}_n$ and letting $m\to +\infty$, we obtain that
\begin{equation*}
    \sum_{n=0}^{+\infty} \|x_{n+1}-x_{n}\|^2 \leq \frac{1}{\delta}(\mathcal{F}_0-\overline{\mathcal{F}}) < +\infty.
\end{equation*}
Therefore, as $n \to +\infty$, $x_{n+1}-x_{n} \to 0$, and so $F(x_n) =\mathcal{F}_n-c\|x_n-x_{n-1}\|^2 \to \overline{\mathcal{F}}$, which means that the sequence $(F(x_n))_{n\in\mathbb{N}}$ is convergent.

Now, we observe that
\begin{align*}
    F(x_n)=\mathcal{F}_n - c\|x_n-x_{n-1}\|^2 \leq \mathcal{F}_n \leq \mathcal{F}_0=F(x_0),
\end{align*}
which implies $x_n\in C_0 =\{x\in C: F(x) \leq F(x_0)\}$. Hence, $(x_n)_{n\in \mathbb{N}}$ is bounded due to the boundedness of $C_0$.

\ref{t:cvg_crit2}: As $\overline{x}$ is a cluster point of the sequence $(x_n)_{n\in \mathbb{N}}$, there exists a subsequence $(x_{k_n})_{n\in \mathbb{N}}$ of $(x_n)_{n\in \mathbb{N}}$ such that $x_{k_n}\to \overline{x}$ as $n \to +\infty$. Then $\overline{x}\in C$ and, since $x_{n+1}-x_n\to 0$, one has $x_{k_n-1}\to \overline{x}$, so as $u_{k_n-1}$ and $v_{k_n-1}$. Since $g +\frac{\beta}{2}\|\cdot\|^2$ is a continuous convex function on an open set $O$ containing $C$, we obtain from \cite[Example~9.14]{Rockafellar1998} that $g$ is locally Lipschitz continuous on $O$. In view of Lemma~\ref{l:upsemicont}, since $x_{k_n}\to \overline{x}$ as $n\to +\infty$, passing to a subsequence if necessary, we can assume that $g_{k_n}\to \overline{g}\in \partial_L g(\overline{x})$ as $n\to +\infty$.    

Replacing $n$ in \eqref{eq:fx+} with $k_n-1$, we have for all $n\in \mathbb{N}$ and all $x\in C$ that
\begin{align}\label{eq:32}
f(x) &\geq f(x_{k_n}) +\langle \nabla h(Au_{k_n-1}), Ax_{k_n}-Ax\rangle -\langle g_{k_n-1}, x_{k_n}-x\rangle \\ \notag 
&\qquad +\frac{1}{2\tau_{k_n-1}}(\|x_{k_n}-v_{k_n-1}\|^2 -\|x-v_{k_n-1}\|^2).
\end{align}
As $\liminf_{n\to +\infty} \tau_n =\overline{\tau} >0$, letting $x=\overline{x}$ and $n\to \infty$, we obtain that $f(\overline{x}) \geq \limsup_{n\to +\infty} {f(x_{k_n})}$. Since $f$ is lower semicontinuous, it follows that $\lim_{n \to +\infty} f(x_{k_n}) =f(\overline{x})$. On the other hand, $\lim_{n \to +\infty} g(x_{k_n}) =g(\overline{x})$ and $\lim_{n \to +\infty} h(Ax_{k_n}) =h(A\overline{x})$ due to the continuity of $g$ and $h$. Therefore, 
\begin{align*}
\lim_{n \to +\infty} F(x_n) =\lim_{n \to +\infty} F(x_{k_n}) 
&=\lim_{n \to +\infty} (f(x_{k_n})+h(Ax_{k_n}) -g(x_{k_n}))\\
&=f(\overline{x})+h(A\overline{x}) -g(\overline{x}) =F(\overline{x}).    
\end{align*}
Next, by letting $n\to +\infty$ in \eqref{eq:32}, for all $x\in C$,
\begin{equation*}
f(x) \geq f(\overline{x}) +\langle \nabla h(A\overline{x}), A\overline{x}-Ax\rangle -\langle \overline{g}, \overline{x}-x\rangle -\frac{1}{2\overline{\tau}}\|x-\overline{x}\|^2,    
\end{equation*}
which can be rewritten as
\begin{align*}
&f(x)+\langle \nabla h(A\overline{x}),Ax-A\overline{x}\rangle -\langle \overline{g}, x-\overline{x}\rangle+\frac{1}{2\overline{\tau}}\|x-\overline{x}\|^2 \\
&\geq f(\overline{x})+\langle \nabla h(A\overline{x}),A\overline{x}-A\overline{x}\rangle -\langle \overline{g}, \overline{x}-\overline{x}\rangle+\frac{1}{2\overline{\tau}}\|\overline{x}-\overline{x}\|^2.
\end{align*}
This means $\overline{x}$ is a minimizer of the function 
$(f+\langle \nabla h(A\overline{x}),A\cdot-A\overline{x}\rangle -\langle \overline{g}, \cdot-\overline{x}\rangle+\frac{1}{2\overline{\tau}}\|\cdot-\overline{x}\|^2)(x)$ over $C$. Hence, $0\in \partial_L(f+\langle \nabla h(A\overline{x}),A\cdot-A\overline{x}\rangle -\langle \overline{g}, \cdot-\overline{x}\rangle+\frac{1}{2\overline{\tau}}\|\cdot-\overline{x}\|^2+\iota_C)(\overline{x}) =\partial_L(f+\iota_C)(\overline{x}) +A^*\nabla h(A\overline{x}) -\overline{g}$, and we must have $\overline{x}\in C\cap \dom f$. Since $\overline{g}\in \partial_L g(\overline{x})$, we deduce that $0\in \partial_L(f+\iota_C)(\overline{x}) +A^*\nabla h(A\overline{x}) -\partial_L g(\overline{x})$, i.e., $\overline{x}$ is a lifted stationary point of \eqref{eq:P}. 

In addition, if we further require that $g$ is strictly differentiable, then Lemma~\ref{lemma:sum} implies that $\overline{x}$ is a stationary point of \eqref{eq:P}. 
\end{proof}

Next, we establish the convergence of the full sequence generated by Algorithm~\ref{algo:extrapolated}. In order to do this, we recall that a proper lower semicontinuous function $G\colon \mathcal{H}\to \left(-\infty, +\infty\right]$ satisfies the \emph{Kurdyka--Łojasiewicz (KL) property} \cite{Kurdyka1998,Loja63} at $\overline{x} \in \dom \partial_L G$ if there are $\eta \in (0, +\infty]$, a neighborhood $V$ of $\overline{x}$, and a continuous concave function $\phi: \left[0, \eta\right) \to \mathbb{R}_+$ such that $\phi$ is continuously differentiable with $\phi' > 0$ on $(0, \eta)$, $\phi(0) = 0$, and, for all $x \in V$ with $
G(\overline{x}) < G(x) < G(\overline{x}) + \eta$, 
\begin{equation*}
\phi'(G(x) -G(\overline{x})) \dist(0, \partial_L G(x)) \geq 1.   
\end{equation*}
We say that $G$ is a \emph{KL function} if it satisfies the KL property at any point in $\dom \partial_L G$. If $G$ satisfies the KL property at $\overline{x} \in \dom \partial_L G$, in which the corresponding function $\phi$ can be chosen as $\phi(t) = c t ^{1 - \theta}$ for some $c \in \mathbb{R}_{++}$ and $\theta \in [0, 1)$, then $G$ is said to satisfy the \emph{KL property at $\overline{x}$ with exponent $\theta$}. The function $G$ is called a \emph{KL function with exponent $\theta$} if it is a KL function and has the same exponent $\theta$ at any $x \in \dom \partial_L G$.

\begin{theorem}[Full sequential convergence]
\label{theorem:global2}  
For problem \eqref{eq:P}, suppose that $F$ is bounded from below on $C$, that the set $C_0 :=\{x\in C: F(x) \leq F(x_0)\}$ is bounded, that $g$ is differentiable on an open set containing $C\cap \dom f$ whose gradient $\nabla g$ is Lipschitz continuous with modulus $\ell_g$ on $C\cap \dom f$, and that $\liminf_{n\to +\infty} \tau_n =\overline{\tau} >0$. Define 
\begin{equation*}
G(x,y) :=F(x) +\iota_C(x) +c\|x-y\|^2,    
\end{equation*}
where $c =\frac{1}{2}(\ell\|A\|^2\overline{\lambda}+\overline{\mu})$, and suppose that $G$ satisfies the KL property at $(\overline{x},\overline{x})$ for every $\overline{x} \in C\cap \dom f$. Then
\begin{enumerate}
\item
The sequence $(x_n)_{n\in\mathbb{N}}$ converges to a stationary point $x^*$ of \eqref{eq:P} and $\sum_{n=0}^{+\infty}\|x_{n+1}-x_n\|<+\infty$.
\item 
Suppose further that $G$ satisfies the KL property with exponent $\theta\in [0,1)$ at $(\overline{x},\overline{x})$ for every $\overline{x} \in C\cap \dom f$. The following statements hold:
\begin{enumerate}
\item
If $\theta =0$, then  $(x_n)_{n\in\mathbb{N}}$ converges to $x^*$ in a finite number of steps. 
\item
If $\theta\in (0,\frac{1}{2}]$, then there exist $\gamma\in \mathbb{R}_{++}$ and $\rho\in \left(0,1\right)$ such that, for all $n\in \mathbb{N}$, $\|x_n-x^*\|\leq \gamma\rho^{\frac{n}{2}}$ and $\lvert F(x_n) -F(x^*) \rvert \leq \gamma\rho^n$.
\item
If $\theta\in (\frac{1}{2},1)$, then there exists $\gamma\in \mathbb{R}_{++}$ such that, for all $n\in \mathbb{N}$, $\|x_n-x^*\|\leq \gamma n^{-\frac{1-\theta}{2\theta-1}}$ and $\lvert F(x_n) -F(x^*)\rvert \leq \gamma n^{-\frac{2-2\theta}{2\theta-1}}$.
\end{enumerate}
\end{enumerate}
\end{theorem}
\begin{proof}
For each $n\in \mathbb{N}$, let $z_n =(x_{n+1},x_n)$. According to Theorem~\ref{theorem:cvg_ex}, we have that, for all $n\in \mathbb{N}$,
\begin{equation*}
G(z_{n+1}) +\delta\|x_{n+2}-x_{n+1}\|^2 \leq G(z_n),    
\end{equation*}
that the sequence $(z_n)_{n\in \mathbb{N}}$ is bounded, that $z_{n+1}-z_n\to 0$ as $n\to +\infty$, and that for every cluster point $\overline{z}$ of $(z_n)_{n\in \mathbb{N}}$, $\overline{z} = (\overline{x}, \overline{x})$, where $\overline{x}\in C\cap \dom f$ is a stationary point of \eqref{eq:P} and $G(z_n) = F(x_{n+1}) +c\|x_{n+1}-x_n\|^2\to F(\overline{x}) =G(\overline{z})$ as $n\to +\infty$.

Let $n\in \mathbb{N}$. It follows from the update of $x_{n+1}$ in Step~2 of Algorithm~\ref{algo:extrapolated} that
\begin{equation*}
0\in \partial_L (f+\iota_C)(x_{n+1}) +\frac{1}{\tau_n}(x_{n+1}-v_n+\tau_n A^*\nabla h(Au_n)-\tau_n \nabla g(x_n)),
\end{equation*}
which implies that
\begin{equation*}
\nabla g(x_n) -A^*\nabla h(Au_n) -\frac{1}{\tau_n}(x_{n+1}-v_n) \in \partial_L (f+\iota_C)(x_{n+1}).
\end{equation*}
Noting that $G(z_n) =(f+\iota_C)(x_{n+1})+h(Ax_{n+1})-g(x_{n+1})+c\|x_{n+1}-x_n\|^2$ and that
\begin{align*}
\partial_L G(z_n) &=\{\partial_L (f+\iota_C)(x_{n+1})+A^*\nabla h(Ax_{n+1})- \nabla g(x_{n+1})+2c(x_{n+1}-x_n)\}  \times \{2c(x_n-x_{n+1})\},
\end{align*}
we obtain
\begin{align*}
\dist(0,\partial_L G(z_n)) &\leq \|\nabla g(x_n) -A^*\nabla h(Au_n) -\frac{1}{\tau_n}(x_{n+1}-v_n)  +A^*\nabla h(Ax_{n+1}) \\
&\qquad -\nabla g(x_{n+1}) +2c(x_{n+1}-x_n)\| +2c\|x_n-x_{n+1}\|\\
&\leq \ell_g\|x_{n+1}-x_n\| +\ell\|A\|\|x_{n+1}-u_n\| +\frac{1}{\tau_n}\|x_{n+1}-v_n\|\\
&\qquad +4c\|x_{n+1}-x_n\|.
\end{align*}
Since $\|x_{n+1}-u_n\| \leq \|x_{n+1}-x_n\|+\lambda_n\|x_n-x_{n-1}\|$ and $\|x_{n+1}-v_n\| \leq \|x_{n+1}-x_n\|+\mu_n\|x_n-x_{n-1}\|$, we derive that
\begin{align*}
\dist(0,\partial_L G(z_n)) &\leq \left(\ell_g + \ell\|A\| + \frac{1}{\tau_n} + 4c \right)\|x_{n+1}-x_n\|\\
&\qquad +\left(\ell\|A\|\lambda_n + \frac{\mu_n}{\tau_n}\right)\|x_n-x_{n-1}\|.
\end{align*}
Since $\liminf_{n\to +\infty} \tau_n = \overline{\tau} > 0$, there exists $n_0 \in \mathbb{N}$ such that, for all $n \geq n_0$, $\tau_n \geq \overline{\tau}/2$. Recalling that $\lambda_n \leq \overline{\lambda}$ and $\frac{\mu_n}{\tau_n} \leq \overline{\mu}$, we have for all $n\geq n_0$ that
\begin{align*}
\dist(0,\partial_L G(z_n)) \leq  \eta_1\|x_{n+1}-x_n\| +\eta_2\|x_n-x_{n-1}\|,
\end{align*}
where $\eta_1 =\ell_g+\ell\|A\|+\frac{2}{\overline{\tau}}+4c$ and $\eta_2 =\ell\|A\|\overline{\lambda}+\overline{\mu}$. Now, the first conclusion follows by applying \cite[Theorem~5.1]{Bo2021} with $I =\{1,2\}$, $\lambda_1 =\frac{\eta_1}{\eta_1+\eta_2}$, $\lambda_2 =\frac{\eta_2}{\eta_1+\eta_2}$, $\Delta_n =\|x_{n+2}-x_{n+1}\|$, $\alpha_n \equiv \delta$, $\beta_n \equiv \frac{1}{\eta_1+\eta_2}$, and $\varepsilon_n \equiv 0$. The remaining conclusions follow a rather standard line of argument as used in \cite{Attouch2007,Bo2021,Li2017}, see also \cite[Theorem~3.11]{BMG22}.
\end{proof}

\begin{remark}[KL property and KL exponents]\label{remark:KLexponent} 
In the preceding theorem, the convergence of the full sequence generated by Algorithm \ref{algo:extrapolated} requires the KL property of the function $G$ with the form that $G(x,y) :=F(x) +\iota_C(x) +c\|x-y\|^2$, where $F$ is the objective function of the model problem \eqref{eq:P}, $C$ is the feasible region of problem \eqref{eq:P} and $c>0$. 
We note that this assumption holds for a broad class of model problem \eqref{eq:P} where $F$ is a semi-algebraic function and $C$ is a semi-algebraic set. More generally, it continues to hold when $F$ is a definable function and $C$ is a definable set (see \cite{Kurdyka1998,BDLS07}). 

As simple illustrations, in our case study in the next section, we will consider the following two classes of functions:
\begin{enumerate}
\item\label{l1l2} 
$F(x) = \varphi(Ax)+\gamma(\|x\|_1-\alpha\|x\|)$, where $\varphi(z) = \frac{1}{2}\|z-b\|^2$ (least square loss) or $\varphi(z) = \|z-b\|_{LL_{2,1}} = \sum_{i=1}^{m}\log \left(1+\lvert z_i-b_i \rvert^2\right)$ (Lorentzian norm loss \cite{Lorentian}), $A \in \mathbb{R}^{m\times d}$, $b\in\mathbb{R}^m$, $\alpha \in \mathbb{R}_+$, and $\gamma \in \mathbb{R}_{++}$.
\item\label{quad}  
$F(x) = \frac{1}{2}x^TMx+u^Tx+r$, where $M$ is an $(d \times d)$ symmetric  matrix , $u \in \mathbb{R}^d$, and $r \in \mathbb{R}$.
\end{enumerate}
Let $G(x,y) :=F(x) +\iota_C(x) +c\|x-y\|^2$, where $c > 0$ and $C$ is a semi-algebraic set in $\mathbb{R}^d$. Then, in both cases, $G$ is definable, and so, it satisfies the KL property at $(\overline{x},\overline{x})$ for all $\overline{x} \in C\cap \dom F$. Moreover, for  case \ref{quad}, if $C$ is further assumed to be a polyhedral set, then as shown in \cite{Li2017} the KL exponent for $G$ is $\frac{1}{2}$, and by Theorem~\ref{theorem:global2}, the proposed algorithm exhibits a linear convergence rate.
\end{remark}

\section{Case studies}
\label{sec:casestudy}

In this section, we provide the numerical results of our proposed algorithm for two case studies: compressed sensing with $L_1-L_2$ regularization, and optimal power flow problem which considers photovoltaic systems placement for a low voltage network. All of the experiments are performed in MATLAB R2021b on a 64-bit laptop with Intel(R) Core(TM) i7-1165G7 CPU (2.80GHz) and 16GB of RAM.

\subsection{Compressed sensing with $L_1-L_2$ regularization}\label{casestudy:case1}

We consider the compressed sensing problem
\begin{align}\label{eq:prob41}
    \min_{x\in\mathbb{R}^d}   \left( \varphi(Ax)+\gamma(\|x\|_1-\alpha\|x\|) \right),
\end{align}
where $A\in\mathbb{R}^{m\times d}$ is an underdetermined sensing matrix of full row rank, $\gamma \in \mathbb{R}_{++}$, and $\alpha \in \mathbb{R}_{++}$. Here, $\varphi$ can be the least square loss function and the Lorentzian norm loss function mentioned in Remark~\ref{remark:KLexponent}.

In our numerical experiments, we let $\alpha=1$ to be consistent with the setting in \cite{Lou2017}. We first start with the least square loss function. By letting $\varphi(z)=\frac{1}{2}\|z-b\|^2$, where $b\in\mathbb{R}^m \setminus \{0\}$, the problem \eqref{eq:prob41} now becomes
\begin{equation}\label{eq:L1L2}
\min_{x\in\mathbb{R}^d} \left( \frac{1}{2}\|Ax-b\|^2+\gamma(\|x\|_1-\|x\|) \right).
\end{equation}
This is known as the regularized least square problem, which has many applications in signal and image processing \cite{Nikolova_2005, Kim_2007,Lou2017}. To solve problem \eqref{eq:L1L2}, we use Algorithm~\ref{algo:extrapolated} with $f =\gamma\|\cdot\|_1$, $h =\varphi$, and $g =\gamma\|\cdot\|$. Then the update of $x_{n+1}$ in Step~2 of Algorithm~\ref{algo:extrapolated} reads as 
\begin{align*}
x_{n+1} = \argmin_{x\in \mathbb{R}^d} (\gamma\|x\|_1 +\frac{1}{2\tau_n}\|x-w_n\|^2) = \prox_{\gamma\tau_n\|\cdot\|_1} (w_n),
\end{align*}
where $w_n =v_n-\tau_n A^*(Au_n-b)+\gamma\tau_n g_n$, and where $g_n\in \partial_L \|\cdot\|(x_n)$ is given by
\begin{equation*}
g_n =\begin{cases}
0 & \text{if $x_n =0$},\\
\frac{x_n}{\|x_n\|} & \text{if $x_n \neq 0$}.
\end{cases}
\end{equation*}
In this case, the proximal operator is the \emph{soft shrinkage} operator \cite{Beck2009}, and so, for all $i=1,\dots,d$,
\begin{equation*}
(x_{n+1})_i =\operatorname{sign}((w_n)_i) \max\{0, \lvert(w_n)_i\rvert-\gamma\tau_n\}.
\end{equation*}
For this test case, we compare our proposed Algorithm~\ref{algo:extrapolated} with the following algorithms:
\begin{itemize}
    \item Alternating direction method of multipliers (ADMM) proposed in \cite{Lou2017}; 
    \item Generalized proximal point algorithm (GPPA) proposed in \cite{An2016}; 
    \item Proximal difference-of-convex algorithm with extrapolation (pDCAe) in \cite{Wen2017}.  
\end{itemize}
Note that the ADMM algorithm uses the $L_1-L_2$ proximal operator which was first proposed in \cite{Lou2017}. For ADMM,  we have $f(x) =\gamma\|x\|_1-\gamma\|x\|$, $h(x) =\varphi(Ax)$, and $g \equiv 0$. For GPPA and pDCAe, we let $f(x) =\gamma\|x\|_1$, $h(x) =\varphi(Ax)$, and $g(x) =\gamma\|x\|$. The parameters of ADMM and pDCA are derived from \cite{Lou2017,Wen2017}. The step size for GPPA and pDCAe are $0.8/\lambda_{\max}(A^TA)$ and $1/\lambda_{\max}(A^TA)$, respectively, where $\lambda_{\max}(M)$ is the maximum eigenvalue of a symmetric matrix $M$. We set $\gamma=0.1$ and run all algorithms, initialized at the origin, for a maximum of 3000 iterations. Note that $\beta =0$ (since $g$ is convex) and $\ell =1$ (since $\nabla \varphi(z) =z-b$). For our proposed algorithm, $\delta=5\times10^{-25}$, $\overline{\lambda}=0.1$, $\overline{\mu}=0.01$, 
$\tau_n=1/\left(2\delta +\ell\|A\|^2\left(2\overline{\lambda}+1\right) +2\overline{\mu}\right)$ with $\|A\|$ being spectral norm, and
\begin{align*}
\lambda_n=\overline{\lambda}\frac{\kappa_{n-1}-1}{\kappa_n},\ \mu_n=\overline{\mu}\tau_n\frac{\kappa_{n-1}-1}{\kappa_n},
\end{align*}
where $\kappa_{-1}=\kappa_0=1$ and $\kappa_{n+1}=\frac{1+\sqrt{1+4\kappa_n^2}}{2}$. Here, we adopt the well-known restarting techniques (see, for example, \cite[Chapter~10]{Beck2017}) and reset $\kappa_{n-1}=\kappa_n=1$ every 50 iterations. Note that this technique has been utilized in several existing work such as \cite{Bo2021,Wen2017}. We generate the vector $b$ based on the same method as in \cite{Lou2017}. In generating the matrix $A$,  we use both randomly generated Gaussian matrices and discrete cosine transform (DCT) matrices. For each cases, we consider different matrix sizes of $m \times d$ with sparsity level $s$ as given in Table~\ref{tab:casestudy1}. For the ground truth sparse vector $x_g$, a random index set is generated and non-zero elements are drawn following the standard normal distribution. The stopping condition for all algorithms is $\frac{\|x_{n+1}-x_n\|}{\|x_n\|}<10^{-8}$. 

\begin{table}[h]
\caption{Test cases for Case study~\ref{casestudy:case1}}
\label{tab:casestudy1}
\begin{center}
\footnotesize
\begin{tabular}{l|c|rrr}
\hline
      Matrix type   & Case & $m$ & $d$  &$s$ \\ \hline
         & 1    & 180 & 640  & 20  \\
Gaussian & 2    & 360 & 1280 & 40  \\
         & 3    & 720 & 2560 & 80  \\
         & 4 & 2880 & 10240 & 320 \\\hline
         & 5    & 180 & 640  & 20  \\
DCT      & 6    & 360 & 1280 & 40  \\
         & 7    & 720 & 2560 & 80  \\
         & 8    & 2880 & 10240 & 320  \\\hline
\end{tabular}
\end{center}
\end{table}

In Table~\ref{tab:resultcase1}, we report the CPU time, the number of iteration, and the function values at termination, the error to the ground truth at termination, averaged over 30 random instances. It can be observed that since the Step~2 involves the calculation of matrix multiplication, the CPU time is significantly increased with the increasing dimension of the matrices. In addition, in terms of running time, objective function values, the number of iterations used, and the error with respect to the ground truth solution (defined as $\frac{\|x_{n+1}-x_g\|}{\|x_g\|}$), our proposed algorithm outperforms ADMM and GPPA in all test cases. Our algorithm also appears to be comparable to pDCAe. Note that our algorithm can be applied to a more general framework than the others.

\begin{table}[h]
\caption{Results of 30 random generated instances for 8 test cases - Least square loss function}
\label{tab:resultcase1}
\begin{center}
\setlength{\tabcolsep}{1pt}
\resizebox{\columnwidth}{!}{
\begin{tabular}{c|cccc|cccc|cccc}
\hline
\multicolumn{1}{l|}{} & \multicolumn{4}{c|}{CPU time (seconds)} & \multicolumn{4}{c|}{Iteration}  & \multicolumn{4}{c}{Error vs ground truth}     \\ \hline
Case                 & ADMM    & GPPA     & pDCAe  & Proposed  & ADMM & GPPA  & pDCAe & Proposed & ADMM      & GPPA      & pDCAe     & Proposed  \\ \hline
1	&0.15&	\textbf{0.02}&	\textbf{0.02}&	\textbf{0.02}&	1803&	487&	\textbf{274}&	406&	5.739E-04&	3.702E-07&	3.505E-07&	\textbf{2.987E-07}\\
2&	0.42&	0.15&	\textbf{0.12}&	0.14&	1583&	449&	\textbf{292}&	325&	2.268E-04&	3.340E-07&	\textbf{1.095E-07}&	2.316E-07\\
3&	2.92&	1.20&	0.87&	\textbf{0.84}&	1471&	417&	300&	\textbf{298}&	2.059E-04&	3.039E-07&	\textbf{6.742E-08}&	2.132E-07\\
4&	61.06&	17.77&	15.54&	\textbf{14.13}&	1415&	380&	311&	\textbf{279}&	1.893E-04&	2.717E-07&	\textbf{4.937E-08}&	1.913E-07\\
5&	0.05&	\textbf{0.01}&	\textbf{0.01}&	\textbf{0.01}&	612&	157&	121&	\textbf{112}&	7.222E-05&	9.690E-08&	8.409E-08&	\textbf{8.026E-08}\\
6&	0.17&	0.07&	\textbf{0.06}&	\textbf{0.06}&	627&	186&	128&	\textbf{112}&	7.095E-05&	2.345E-05&	\textbf{5.180E-08}&	6.383E-08\\
7&	1.19&	0.50&	0.36&	\textbf{0.31}&	634&	170&	131&	\textbf{113}&	7.116E-05&	1.777E-06&	\textbf{4.240E-08}&	6.303E-08\\
8&	29.92&	8.11&	7.49&	\textbf{6.47}&	721&	181&	155&	\textbf{133}&	6.908E-05&	3.984E-05&	\textbf{2.742E-08}&	6.092E-08\\
 \hline
\end{tabular}
}
\end{center}
\end{table}

Next, we consider the case of Lorentzian norm loss function by letting $\varphi(z)=\|z-b\|_{LL_{2,1}}$. Lorentzian norm can be useful in robust sparse signal reconstruction \cite{Lorentian}. In this case, the optimization problem \eqref{eq:prob41} becomes
\begin{align}\label{eq:L1L2_Lorentzian}
    \min_{x\in\mathbb{R}^d}   \left( \|Ax-b\|_{LL_{2,1}}+\gamma(\|x\|_1-\|x\|) \right).
\end{align}
We note that 
\begin{align*}
    \nabla \varphi(z)=\left(\frac{2(z_1-b_1)}{1+ \lvert z_1-b_1 \rvert ^2},\dots,\frac{2(z_m-b_m)}{1+\lvert z_m-b_m \rvert^2}\right)^T. 
\end{align*}
is Lipschitz continuous with modulus $\ell =2$. Since the loss function is now nonconvex and the pDCAe algorithm in \cite{Wen2017} requires a convex loss function, pDCAe is not applicable in this case. Moreover, the ADMM algorithm in \cite{Lou2017} is also not directly applicable due to the presence of the Lorentzian norm. Therefore, we compare our method with the GPPA only.  For GPPA, we let $h(x)=\varphi(Ax)$. The stepsize for GPPA is $\tau=0.8/(2\lambda_{\max}(A^TA))$. For this case, we set $\gamma=0.001$ and run the GPPA and our proposed algorithm, which are both initialized at the origin, for a maximum of 4000 iterations. The remaining parameters of our algorithm are set to the same values as before. We also use 30 random instances of the previous 8 test cases. The results are presented in Table~\ref{tab:Lorenzt}. It can be seen from Table~\ref{tab:Lorenzt} that the proposed algorithm outperforms GPPA in this case.

\begin{table}[h]
\caption{Results of 30 random generated instances for 8 test cases - Lorentzian norm loss function}
\label{tab:Lorenzt}
\footnotesize
\centering
\setlength{\tabcolsep}{4pt}
\begin{tabular}{c|cc|cc|cc} 
\hline
\multicolumn{1}{l|}{} & \multicolumn{2}{c|}{CPU time (seconds)} & \multicolumn{2}{c|}{Iteration} & \multicolumn{2}{c}{Error vs ground truth}  \\ 
\hline
Case                 & GPPA   & Proposed                      & GPPA & Proposed               & GPPA   & Proposed                          \\
\hline
1	&0.36	&\textbf{0.30}	&2104	&\textbf{1720}		&2.865E-03&	\textbf{2.863E-03}\\
2	&2.94&	\textbf{2.44}&	2282&	\textbf{1870}		&4.043E-03&	\textbf{3.132E-03}\\
3	&17.15&	\textbf{14.06}&	2369&	\textbf{1936}		&3.168E-03&	\textbf{3.166E-03}\\
4	&277.46	&\textbf{225.60}&	2438&	\textbf{1993}	&3.356E-03&	\textbf{3.354E-03}\\
5	&0.36	&\textbf{0.30}	&2148	&\textbf{1765}	&1.425E-03	&\textbf{1.416E-03}\\
6	&3.00	&\textbf{2.47}	&2347	&\textbf{1922}	&2.134E-03	&\textbf{1.169E-03}\\
7	&16.85	&\textbf{13.59}	&2334	&\textbf{1908}	&1.213E-03	&\textbf{1.205E-03}\\
8	&260.61	&\textbf{220.55}	&2440	&\textbf{2064}	&2.269E-03	&\textbf{2.261E-03}\\
\hline
\end{tabular}
\end{table}

\subsection{Optimal power flow considering photovoltaic systems placement}\label{subsec:dcopf}

Optimal power flow (OPF) is a well-known problem in power system engineering \cite{Abdi2017}. The integration of many distributed energy resources (DERs) such as photovoltaic systems, has become increasingly popular in modern smart grid \cite{Wankhede2020}, leading to the needs of developing more complicated OPF models considering the DERs. Metaheuristic algorithms are popular in solving OPF, and they have also been applied to solve the OPF with DERs integration \cite{Shaheen2019, Khaled2017}. However, the drawbacks of the metaheuristic algorithms are that the convergence proof cannot be established, and their performances are not consistent \cite{Ezugwu2019}. Difference-of-convex programming has also been successfully applied to solve the OPF problem in \cite{Merkli2018}, although DERs are not considered. Motivated by the aforementioned results, in this work we try to applied our proposed algorithm to solve the OPF in a low voltage network, which includes optimizing the placement of photovoltaic (PV) systems. We formulate two models which are based on the Direct Current OPF (DC OPF) \cite{Kargarian2018}, and Alternating Current OPF (AC OPF) \cite{Farivar2013}. To the best of the authors' knowledge, this is the first time a proximal algorithm is used to solve an DER-integrated OPF with a difference-of-convex formulation, considering PV systems placement. The objective function aims at minimizing the cost of the conventional generator, which is a diesel generator in this case study, while maximizing the PV-penetration, which is defined as the ratio of the power generated by the PV systems divided by the total demand. The network considered in this case study is illustrated in Figure~\ref{fig:F1}, which consists of 14 buses. This case study is taken from a real low voltage network in Victoria, Australia. Currently, there are demands at bus 1, 3, 4, 6, 8, 9, 13, and 14. There are 6 PV systems at bus 1, 2, 4, 5, 7, and 8 with a capacity of 800 kW. A 5000 kW diesel generator is connected to bus 11. All of the parameters and decision variables in this case study are presented in Table~\ref{tab:paramsvars}. The cost of the current situation (before optimization is performed) is based on the cost of active power withdrawn from the generator, plus the installation cost of the PV systems. To determine this initial cost, the amount of active power generated by the generator is determined via DIgSILENT Power Factory 2021. After that, the cost of active power is calculated by the expression $\sum_{i\in M}(a(P_{i}^{G})^2+bP_{i}^{G}+c)$, plus the installation cost of the six PV systems.

We first formulate the OPF problem with PV, which is based on the DC OPF, as follows
\begin{subequations}
\begin{align}
    \min~~ & \left( \sum_{i\in N}CX_i + \sum_{i\in M}\left(a(P_{i}^{G})^2+bP_{i}^{G}+c\right) - \frac{\sum_{i\in N} P^{PV}_i}{\sum_{i\in N} D_i}\right)\label{eq:objOPF0}\\
    \text{subject to~~} & P_{ij}=b_{ij}(\theta_i-\theta_j), \quad \forall i,j \in N\label{eq:OPFflow}\\
    & \theta_{11}=0\label{eq:slackbus}\\   
    &\sum_{j\in N, j \neq i}P_{ij}=P_{i}^{PV}+P_{i}^{G}-D_i, \quad \forall i \in M \label{eq:flowinout}\\
    &\sum_{j\in N, j \neq i}P_{ij}=P_{i}^{PV}-D_i, \quad \forall i \in N \setminus M \label{eq:flowinout2}\\
    &\frac{\sum_{i\in N}P_i^{PV}}{\sum_{i\in N}D_i}\geq 0.5 \label{eq:penrate}\\
    &\lvert P_{ij} \rvert\leq \overline{P}, \quad \forall i,j \in N \label{eq:thermal}\\
     &0\leq P_{i}^{PV}\leq X_i\overline{P^{PV}}, \quad \forall i \in N \label{eq:existX}\\
     &0\leq P_{i}^{G}\leq \overline{P^G}, \quad \theta_i \in [0,2\pi], \quad \forall i \in M \label{eq:bounds}\\
     &X_i \in \{0,1\}, \quad \forall i \in N.
\end{align}
\end{subequations}
We see that for any $i \in N$, if $X_i\in [0,1]$, then $X_i - X_i^2 = X_i(1-X_i) \geq 0$. Therefore,
\begin{align*}
    &(\forall i\in \{1,\dots, N\},\quad X_i \in \{0,1\}) \\
    &\iff (\forall i\in \{1,\dots, N\},\ X_i \in [0,1] \text{~and~} \sum_{i\in N} \left(X_i^2-X_i \right)\geq 0).
\end{align*}
Taking into account of the above equivalence, a plausible alternative optimization model for the OPF problem with PV is as follows
\begin{subequations}
\begin{align}
    \min~~ & \left(\sum_{i\in N}CX_i + \sum_{i\in M}\left(a(P_{i}^{G})^2+bP_{i}^{G}+c\right) - \frac{\sum_{i\in N} P^{PV}_i}{\sum_{i\in N} D_i} \notag \right.\\
    & \left.-\gamma \sum_{i\in N}(X_i^2 - X_i) \right)\label{eq:objOPF1}\\
    \text{subject to~~} &\eqref{eq:OPFflow} \to \eqref{eq:bounds}\\
        & X_i \in [0,1], \quad \forall i \in N.
\end{align}
\end{subequations}
The objective function~\eqref{eq:objOPF1} aims at minimizing the installation cost and the generation cost of the diesel generator and maximizing the PV penetration, which is defined as $\frac{\sum_{i\in N} P^{PV}_i}{\sum_{i\in N} D_i}$ \cite{Hoke2013}, the parameter $\gamma >0$ which serves as a Lagrangian multiplier for the discrete constraints $X_i \in \{0,1\}$. 

With this reformulation, the objective function \eqref{eq:objOPF1} now becomes a difference-of-convex function. Constraint \eqref{eq:OPFflow} describes the relationship between the power flow from one bus to another and their corresponding phasor angles, constraint \eqref{eq:slackbus} defines the voltage angle at the slack bus, which is the bus connected to the diesel generator, constraint \eqref{eq:flowinout} and \eqref{eq:flowinout2} define the power flow in and out of any buses, constraint \eqref{eq:penrate} ensures that the PV penetration rate is at least 50 percent, constraint \eqref{eq:thermal} defines the transmission limits of the transmission lines, and constraint \eqref{eq:existX} makes sure that the solar power only exists at a bus when there is a PV system at that bus. Finally, constraint \eqref{eq:bounds} defines the boundaries of the remaining decision variables. All of the constraints form the feasible set $S$. This problem takes the form of \eqref{eq:P} with $f =\iota_S$, $h =\sum_{i\in N}CX_i + \sum_{i\in M}\left(a(P_{i}^{G})^2+bP_{i}^{G}+c\right)  - \frac{\sum_{i\in N} P^{PV}_i}{\sum_{i\in N} D_i}$, and $g =\gamma \sum_{i \in N}(X_i^2 - X_i)$. By Remark~\ref{remark:KLexponent}\ref{quad} and Theorem~\ref{theorem:global2}, in this case the proposed algorithm converges with a linear rate. The update of $x_{n+1}$ in Algorithm~\ref{algo:extrapolated} becomes
\begin{align*}
    x_{n+1} = \argmin_{x\in S} \|x-v_n+\tau_n\nabla h(u_n)-\tau_n \nabla g(x_n)\|^2.
\end{align*}
Here,
\begin{align*}
x &= [P^{PV}_1,\dots,P^{PV}_{14}, P^G_{11}, X_1,\dots,X_{14}, \theta_1,\dots,\theta_{14},\\ 
&\qquad P_{1,1},\dots,P_{1,14}, \dots,P_{14,1},\dots,P_{14,14}]^T.
\end{align*}
This step is solved by MATLAB's \texttt{quadprog} command. Noting that $\beta =0$ (since $g$ is convex) and $\ell =2a$, the parameters are set as follows: $\delta=5\times10^{-25}$, $\overline{\lambda}=0.1$, $\overline{\mu}=0.01$, $\tau_n=1/\left(2\delta +\ell\left(2\overline{\lambda}+1\right) +2\overline{\mu}\right)$, $\mu_n=\overline{\mu}\tau_n$, and $\lambda_n$ is chosen in the same way as in Section~\ref{casestudy:case1}. The performance of the proposed algorithm is compared with the the GPPA, and the pDCAe, as illustrated in Table~\ref{tab:compareDCOPF}. We use the step size $\tau_n=0.8/\ell$ for GPPA, and $\tau_n =1/\ell$ for pDCAe. The maximum number of iteration is 1000, and the stopping condition is the same as the one used in Section~\ref{casestudy:case1}. 

We test all algorithms for 30 times, at each time we use a random starting point between the upper bound and the lower bound of the variables. The mean objective function values, and the best objective function values found by all algorithms are reported in Table~\ref{tab:compareDCOPF}. Although the proposed algorithm, on average, needs more iterations than the remaining ones, it can find a better solution. The mean objective function value found by our algorithm is also better than the ones found by the other algorithms. Our algorithm is also comparable to the GPPA and the pDCAe in terms of average CPU time.

\begin{table}[h]
\centering
\caption{Comparison of GPPA, pDCAe, and the proposed algorithm on 30 runs of the DC OPF model}
\label{tab:compareDCOPF}
\footnotesize
\begin{tabular}{c|ccc} 
\hline
Algorithm     & GPPA    & pDCAe   & Proposed  \\ 
\hline
Mean objective function value  & 3.724581 & 3.719692 & \textbf{3.706267}   \\
Best objective function value & 1.920925 & 1.920924 & \textbf{1.920922}   \\
Mean iteration number               & \textbf{3}       & 4       & 5        \\
Mean CPU time (seconds)          & \textbf{0.08}    & 0.11    & 0.12      \\
\hline
\end{tabular}
\end{table}
The details of the best solution found by our algorithm are shown in Figure~\ref{fig:F1}.

\begin{figure}[H] 
\centering
\includegraphics[width=0.5\columnwidth]{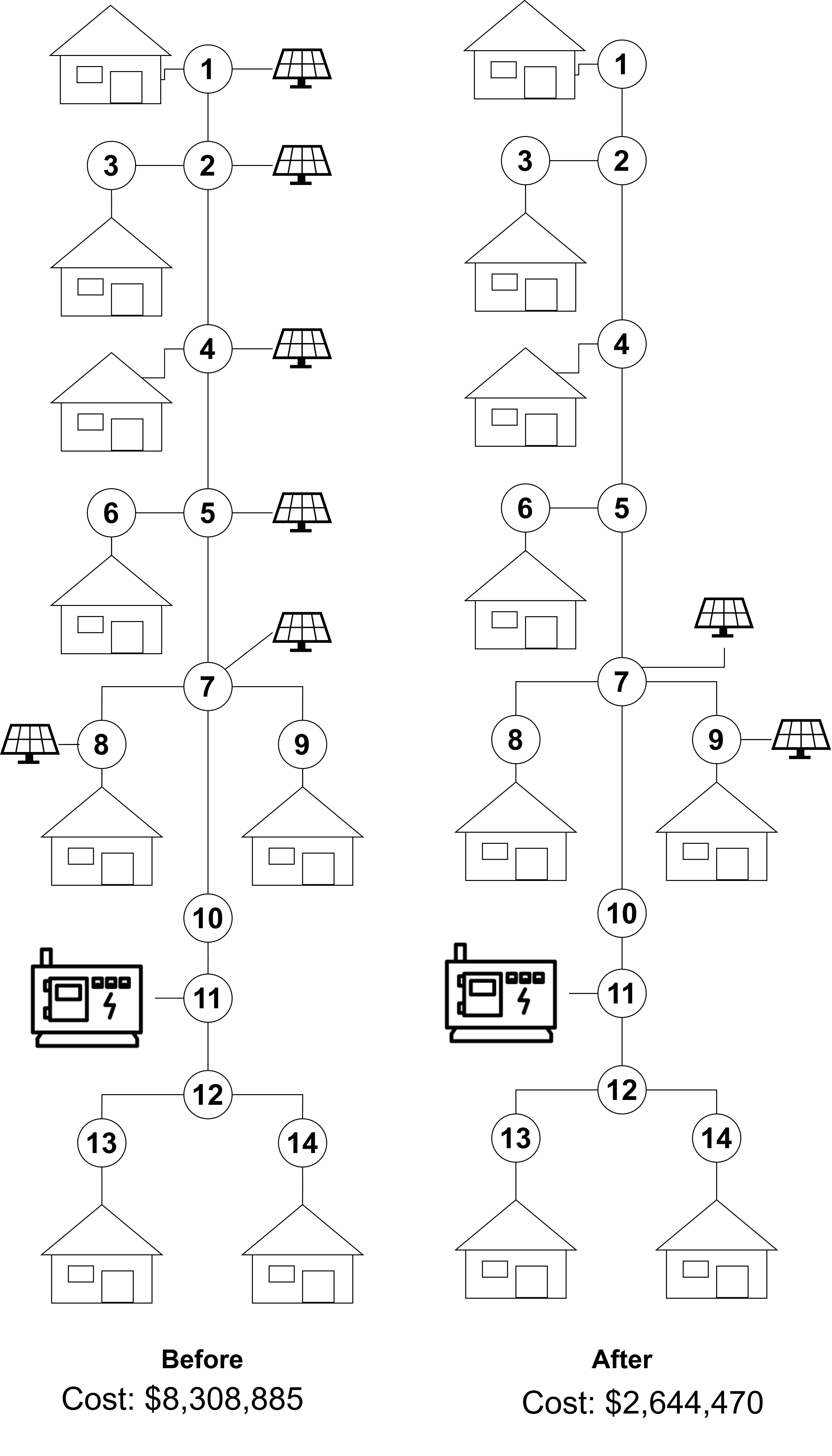}
\caption{Best solution found in Case study~\ref{subsec:dcopf}, together with the total cost.}
\label{fig:F1}
\end{figure}

Now we consider the case of AC OPF model. The formulation is based on the \emph{branch flow model} given in \cite{Farivar2013}. Firstly, the network is treated as a directed graph, as shown in Figure~\ref{fig:tree}. 

\begin{figure}[H]
\centering
\includegraphics[width=0.3\columnwidth]{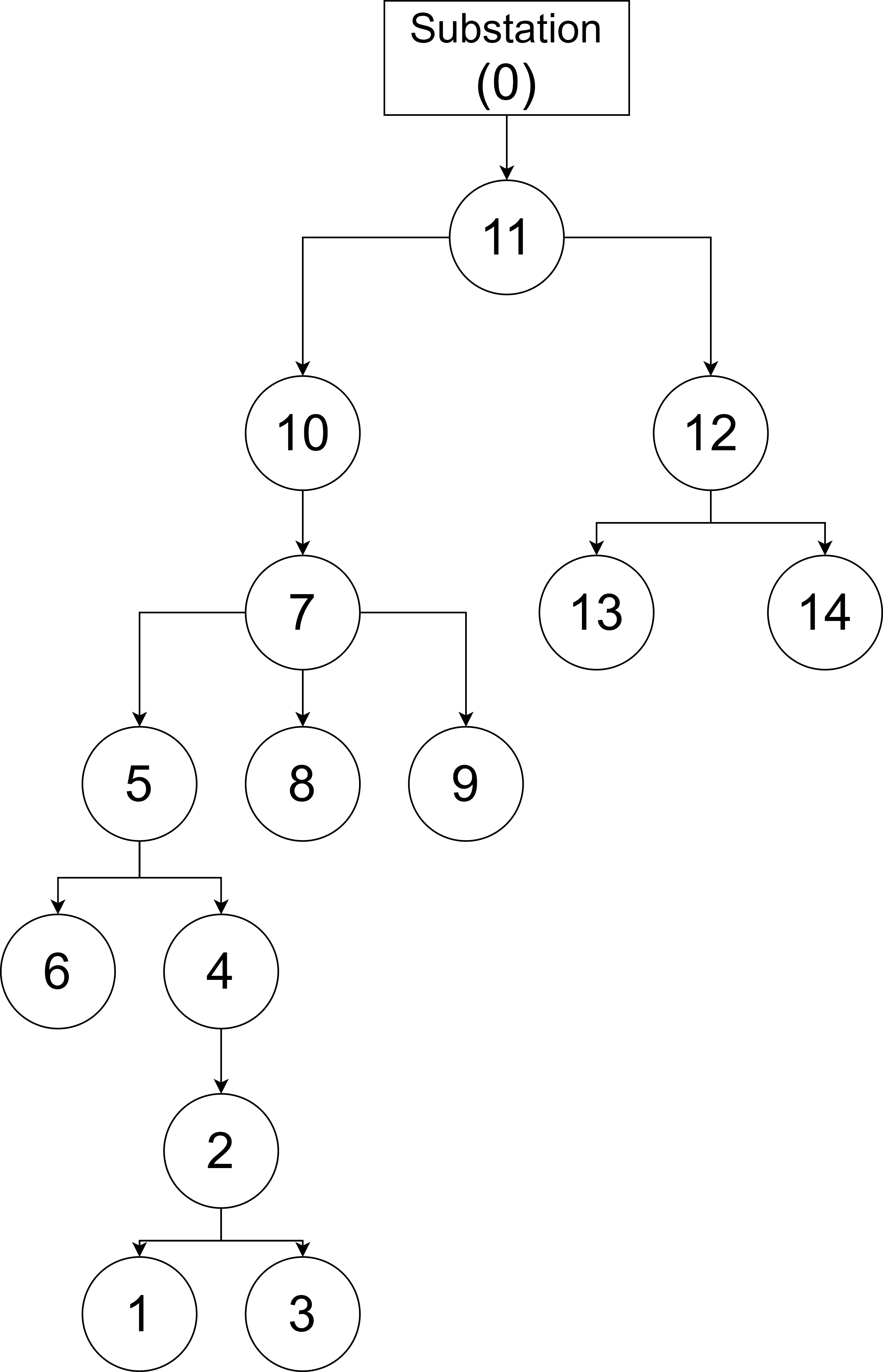}
\caption{Directed graph representation of the network.}
\label{fig:tree}
\end{figure}

We denote a directed link by $(i,j)$ or $i\to j$ if it points from bus $i$ to bus $j$, and the set of all directed links by $E$. Next, the formulation is given as follows,

\begin{subequations}
\begin{align}
    \min~ & \left( \sum_{i\in N}CX_i + \sum_{i\in M}\left(a(P_{i}^{G})^2+bP_{i}^{G}+c\right) - \frac{\sum_{i\in N} P^{PV}_i}{\sum_{i\in N} D_i} \notag \right. \\
    &\qquad \left. -\gamma \sum_{i\in N}(X_i^2 - X_i) \right)\label{eq:objACOPF}\\
    \text{s.t.~~}  &\hat{I}_{ij}= \lvert I_{ij} \rvert ^2, \quad \forall(i,j) \in E\\
    & v_{i}=\lvert V_{i} \rvert ^2,\quad \forall i \in N\\
    & P_{0,11}=Q_{0,11}=0\\
    &P_{ij}+P^{PV}_{j}+P^G_{j}-D_j=\sum_{k\in N: j\to k}P_{jk},~\forall (i,j) \in E,~j \in M \label{con:acflow1}\\
    &Q_{ij}+Q^{PV}_{j}+Q^G_{j}-D^Q_j=\sum_{k\in N: j\to k}Q_{jk},~\forall (i,j) \in E,~j \in M \\
    &P_{ij}+P^{PV}_{j}-r_{ij}\hat{I}_{ij}-D_j=\sum_{k \in N: j\to k}P_{jk},~\forall (i,j) \in E,~j\notin M \\
    &Q_{ij}+Q^{PV}_{j}-\mathcal{X}_{ij}\hat{I}_{ij}-D^Q_j=\sum_{k \in N: j\to k}Q_{jk},~\forall (i,j) \in E,~j \notin M \label{con:acflow2}\\
    &\frac{\sum_{j\in N}P_j^{PV}}{\sum_{j\in N}D_j}\geq 0.5 \label{con:penrateac}\\
    & 0\leq P^{PV}_j \leq X_j \overline{P^{PV}}, \quad \forall j \in N \label{con:acpv1}\\
    & 0 \leq Q^{PV}_j \leq X_j \overline{Q^{PV}}, \quad \forall j \in N \label{con:acpv2}\\
    &v_j=v_i-2(r_{ij}P_{ij}+\mathcal{X}_{ij}Q_{ij})+(r_{ij}^2+\mathcal{X}_{ij}^2)\hat{I}_{ij}, ~\forall (i,j) \in E \label{con:acvrelation}\\
    & \hat{I}_{ij}v_i= P_{ij}^2+Q_{ij}^2, \quad \forall (i,j) \in E \label{con:physicalmean}\\
    & \underline{V}^2 \leq v_i \leq \overline{V}^2, \quad \forall i \in N \label{con:bound1}\\
     & \underline{I}^2 \leq \hat{I}_{ij} \leq \overline{I}^2, \quad \forall (i,j) \in E\\
    &\lvert P_{ij} \rvert\leq \overline{P}, \quad \forall i,j \in E\\
    &\lvert Q_{ij} \rvert \leq \overline{Q}, \quad \forall i,j \in E\\
    &0\leq P_{j}^{G}\leq \overline{P^G}, \quad \forall j \in M\\
     &0\leq Q_{j}^{G}\leq \overline{Q^G}, \quad \forall j \in M\\
    &X_j \in [0,1], \quad \forall j \in N \label{con:bound2}
\end{align}
\end{subequations}

The main differences between the AC OPF model and the DC OPF model are that the AC OPF model has a nonconvex feasible set, and that it also accounts for the loss in the network as well as the reactive power. Consequently, AC OPF is more accurate than DC OPF in practice \cite{Frank2016}, and due to its nonconvexity, it is also more challenging to solve \cite{Low2014}. Constraints \eqref{con:acflow1}~$\to$~\eqref{con:acflow2} define the power flow in any directed links. Constraint \eqref{con:penrateac} ensures that the PV penetration rate is at least 50 percent. Constraints \eqref{con:acpv1} and \eqref{con:acpv2} ensure that the active and reactive power from PV systems only exist at a bus if and only if there is a PV system at that bus. Constraint \eqref{con:acvrelation} describes the relationship between the voltage of any two bus in a directed link. Constraint \eqref{con:physicalmean} is a nonconvex constraint ensuring that the solution have physical meaning. Finally, constraints \eqref{con:bound1}~$\to$~\eqref{con:bound2} define the boundaries of the decision variables. The update of $x_{n+1}$ is also the same as before. For this case,
\begin{align*}
      x =[&P^G_{11},Q^G_{11}, v_1, \dots, v_{14},\hat{I}_{0,11}, \dots,  \hat{I}_{2,1},  P_{0,11}, \dots,  P_{2,1},
      Q_{0,11}, \dots, Q_{2,1}, \\
      &P^{PV}_{1}, \dots, P^{PV}_{14},  Q^{PV}_{1}, \dots,Q^{PV}_{14}, X_{1}, \dots, X_{14}  ]^T.
\end{align*}

We also perform the same numerical experiment as in the DC OPF case. However, the pDCAe is not applicable in this case, so we compare our algorithm with the GPPA only. The parameters of GPPA and our proposed algorithm are set to the same values as those used for the DC OPF model. Due to the nonconvex constraint, MATLAB's \texttt{fmincon} is used to solve the subproblem in Step~2 instead of \texttt{quadprog}. The results are shown in Table~\ref{tab:compareACOPF}.

\begin{table}[h]
\centering
\caption{Comparison of GPPA and the proposed algorithm on 30 runs of the AC OPF model}
\label{tab:compareACOPF}
\footnotesize
\begin{tabular}{c|cc} 
\hline
Algorithm    & GPPA   & Proposed  \\ 
\hline
Mean objective function value 	&3.492971&	\textbf{3.416897}
   \\
Best objective function value  & 1.920924 & \textbf{1.920923}   \\
Mean iteration number                & 33      & \textbf{20}       \\
Mean CPU time (seconds)          & 152.69    & \textbf{109.20}    \\
\hline
\end{tabular}
\end{table}

Table~\ref{tab:compareACOPF} shows that our proposed algorithm takes less time and fewer iterations than the GPPA to converge. The best solution found by our algorithm in this case is also the same as the one found in the DC OPF model.

It can be seen that for both DC OPF and AC OPF, two PV systems need to be installed at bus 7 and bus 9, the remaining demands can be supplied by the generator, and the demands are satisfied by the power flows. Although the mathematical model aims at maximizing the PV penetration, drawing power from the diesel generator is still more economical due to the high installation cost of the PV systems. The solution significantly reduces the cost by approximately $70\%$ from the original situation. This can serve as a proof of concept for future research.

\section{Conclusion}\label{sec:conclusion}

We have proposed an extrapolated proximal subgradient algorithm for minimizing a class of structured nonconvex and nonsmooth optimization problems. Our algorithm allows less restriction on the smoothness and convexity requirements for establishing convergence proof. In addition, our choice of the extrapolation parameters is flexible enough to cover the popular one used in FISTA and its variants. The convergence of the whole sequence generated by our algorithm is proved via the abstract convergence framework given in \cite{Bo2021}. The proposed algorithm exhibits very competitive results in terms of numerical experiments which are performed on a compressed sensing problem with nonconvex $L_1-L_2$ regularization, compared with some existing algorithms. We have also applied this algorithm to solve an OPF problem considering PV placement, which serves as a proof of concept for future works.

\subsection*{Acknowledgements}

The research of TNP was supported by Henry Sutton PhD Scholarship Program from Federation University Australia. The research of MND benefited from the FMJH Program Gaspard Monge for optimization and operations research and their interactions with data science, and was supported by a public grant as part of the Investissement d'avenir project, reference ANR-11-LABX-0056-LMH, LabEx LMH. The research of GL was supported by Discovery Project 190100555 from the Australian Research Council.

\appendix
\section{Data of Case study~\ref{subsec:dcopf}}\label{secA1}

In Case study~\ref{subsec:dcopf}, we use a base power of 100 MVA, and a base voltage of 22 kV. All of the parameters are converted into Per Unit (pu) values in the calculation. Readers can refer to \cite[Chapter~2]{Weedy2012} for a detailed tutorial on the Per Unit system. The active and reactive power demand are given in Table~\ref{table:Demand}. The other technical parameters of the system including susceptance, resistance, and reactance of the lines are given in Table~\ref{table:susceptance}, Table~\ref{table:resistance}, and Table~\ref{table:reactance}, respectively.
\begin{table}[H]
\centering
\caption{Parameters and variables of Case study~\ref{subsec:dcopf}}
\label{tab:paramsvars}
\footnotesize
\begin{tabular}{l| p{0.45\columnwidth} |p{0.27\columnwidth}}
\hline
\multicolumn{1}{c|}{Parameters}          & \multicolumn{1}{|c}{Description}                                                                                                                                      & \multicolumn{1}{|c}{Values}                                                                                                    \\ \hline
$N$                 & Set of buses                                                                                                                                                          & $\{1,2,\dots,14\}$                                                                                          \\
$M$                 & Set of buses that are connected to diesel generators, $M\subseteq N$                                                                                                  & $\{11\}$                                                                                                  \\
$E$ & Set of directed links & 
  
    $\{(0,11),(11,10),\dots,
    $\\
    & & $(2,1)\}$\\
$D_i$               & Active power demand at bus $i$                                                                                                                                                     & See Table~\ref{table:Demand}                                                                        \\
$D^Q_i$               & Reactive power demand at bus $i$                                                                                                                                                     & See Table~\ref{table:Demand}                                                                        \\

$b_{ij}$            & Susceptance value of the line connecting bus $i$ and bus $j$                                                                                                          & See Table~\ref{table:susceptance}                                                                                          \\
$r_{ij}$            & Resistance value of the line connecting bus $i$ and bus $j$                                                                                                          & See Table~\ref{table:resistance}                                                                                          \\
$\mathcal{X}_{ij}$            & Reactance value of the line connecting bus $i$ and bus $j$                                                                                                          & See Table~\ref{table:reactance}                                                                                          \\
$C$                 & Unit installation cost of a PV at bus $i$                                                                                                                                  & 1 (1 unit = \$1040000)                                                                                                  \\
$a,b,c$             & Coefficients associated with the cost of diesel generator. These coefficients for a diesel generator are derived from \cite{Kusakana2015,Fodhil2019} & $0.246$, $0.084$,  $0.433$ \\
$\overline{P^{PV}}$ & Active power capacity of PVs                                                                                                                                                 & 800 kW (0.008 pu)                                                                                                    \\
$\overline{Q^{PV}}$ & Reactive power capacity of PVs                                                                                                                                                 & 300 kW (0.003 pu)                                                                                                    \\
$\overline{P^G}$    & Active power capacity of diesel generator                                                                                                                                          & 5000 kW (0.05 pu)                                                                                                   \\
$\overline{Q^G}$    & Reactive power capacity of diesel generator                                                                                                                                          & 3000 kW (0.03 pu)                               \\
$\overline{P},\overline{Q}$ & Transmission limits of lines                                                                                                                                          & 3000 kW (0.03 pu)     \\

$\overline{V}, \underline{V} $ & Voltage limits                                                                                                                                          & 1.05 pu, 0.95 pu     \\
$\overline{I}, \underline{I} $ & Current limits                                                                                                                                          & 2 pu , 0 pu     \\
$\gamma$ & Relaxation parameter & 1
\\ \hline
\multicolumn{1}{c|}{Variables}           &                                                                                                                                                                       &                                                                                                           \\ \hline
$P_{i}^{PV}$        & Active power generated by a PV system at bus $i$, $i\in N$ &   \\

$Q_{i}^{PV}$        & Reactive power generated by a PV system at bus $i$, $i\in N$ &   \\

$P_{i}^{G}$         & Active power generated by diesel generator at bus $i$, $i \in M$                                                                                                                        &                                                                                                           \\
$Q_{i}^{G}$         & Reactive power generated by diesel generator at bus $i$, $i \in M$                                                                                                                        &                                                                                                           \\
$X_i $            & 1 if there is a PV system needed at bus $i$, and 0 otherwise, $i\in N$                                                                                                           &                                                                                                           \\
$V_i $            & Nodal voltage of bus $i$, $i\in N$                                                                                                           &                                                                                                           \\
$I_{ij} $            & Current between bus $i$ and bus $j$, \quad $i, j\in N$, \quad $i \neq j$                                                                                                           &                                                                                                           \\
$\theta_i$          & Voltage angle of bus $i$, $i\in N$                                                                                                                                              &                                                                                                           \\
$P_{ij}$            & Active power flow between bus $i$ and bus $j$, \quad $i, j\in N$, \quad $i \neq j$                                                                                                                                &                                                                                                           \\
$Q_{ij}$            & Reactive power flow between bus $i$ and bus $j$, \quad $i, j\in N$, \quad $i \neq j$                                                                                                                                &                                                                                                           \\
\hline
\end{tabular}
\end{table}

\begin{sidewaystable}

\centering
\caption{Active and Reactive power demand}\label{table:Demand}
\footnotesize
\setlength{\tabcolsep}{4pt}
\begin{tabular}{p{0.05\columnwidth}cccccccccccccc} 
\hline
Bus    & 1     & 2   & 3     & 4     & 5   & 6     & 7   & 8     & 9     & 10  & 11  & 12  & 13    & 14     \\ 
\hline
$D_i$& 7.91E-03 &	0	&2.81E-03	&3.40E-03	&0	&3.05E-03&	0&	3.32E-03&	5.90E-03&	0	&0	&0&	2.12E-03&	2.64E-03
  \\
$D^Q_i$ & 1.98E-03&	0&	7.04E-03&	8.51E-03&0&	7.64E-04
&0&	8.32E-03&	1.48E-03&	0&	0&	0&	5.32E-03&	6.63E-04

  \\
\hline
\end{tabular}
\vspace{2\baselineskip}
\caption{Susceptance $b_{ij}$}\label{table:susceptance}
\setlength{\tabcolsep}{1pt}
\begin{tabular}{ccccccccccccccc} 
\hline
Bus & 1         & 2         & 3         & 4         & 5         & 6         & 7         & 8         & 9         & 10        & 11        & 12        & 13        & 14         \\ 
\hline
1   & -9.98E+02 & 9.98E+02  & 0         & 0         & 0         & 0         & 0         & 0         & 0         & 0         & 0         & 0         & 0         & 0          \\
2   & 9.98E+02  & -2.60E+03 & 4.97E+02  & 1.11E+03  & 0         & 0         & 0         & 0         & 0         & 0         & 0         & 0         & 0         & 0          \\
3   & 0         & 4.97E+02  & -4.97E+02 & 0         & 0         & 0         & 0         & 0         & 0         & 0         & 0         & 0         & 0         & 0          \\
4   & 0         & 1.11E+03  & 0         & -4.35E+03 & 3.24E+03  & 0         & 0         & 0         & 0         & 0         & 0         & 0         & 0         & 0          \\
5   & 0         & 0         & 0         & 3.24E+03  & -4.79E+03 & 5.72E+02  & 9.77E+02  & 0         & 0         & 0         & 0         & 0         & 0         & 0          \\
6   & 0         & 0         & 0         & 0         & 5.72E+02  & -5.72E+02 & 0         & 0         & 0         & 0         & 0         & 0         & 0         & 0          \\
7   & 0         & 0         & 0         & 0         & 9.77E+02  & 0         & -4.00E+03 & 6.92E+02  & 9.26E+02  & 1.41E+03  & 0         & 0         & 0         & 0          \\
8   & 0         & 0         & 0         & 0         & 0         & 0         & 6.92E+02  & -6.92E+02 & 0         & 0         & 0         & 0         & 0         & 0          \\
9   & 0         & 0         & 0         & 0         & 0         & 0         & 9.26E+02  & 0         & -9.26E+02 & 0         & 0         & 0         & 0         & 0          \\
10  & 0         & 0         & 0         & 0         & 0         & 0         & 1.41E+03  & 0         & 0         & -2.27E+03 & 8.64E+02  & 0         & 0         & 0          \\
11  & 0         & 0         & 0         & 0         & 0         & 0         & 0         & 0         & 0         & 8.64E+02  & -3.85E+03 & 2.99E+03  & 0         & 0          \\
12  & 0         & 0         & 0         & 0         & 0         & 0         & 0         & 0         & 0         & 0         & 2.99E+03  & -7.10E+03 & 2.08E+03  & 2.04E+03   \\
13  & 0         & 0         & 0         & 0         & 0         & 0         & 0         & 0         & 0         & 0         & 0         & 2.08E+03  & -2.08E+03 & 0          \\
14  & 0         & 0         & 0         & 0         & 0         & 0         & 0         & 0         & 0         & 0         & 0         & 2.04E+03  & 0         & -2.04E+03  \\
\hline
\end{tabular}
\end{sidewaystable}

\begin{sidewaystable}
\centering
\caption{Resistance $r_{ij}$}\label{table:resistance}
\footnotesize
\setlength{\tabcolsep}{1pt}
\begin{tabular}{ccccccccccccccc} 
\hline
Bus & 1        & 2        & 3        & 4        & 5        & 6        & 7        & 8        & 9        & 10       & 11       & 12       & 13       & 14        \\ 
\hline
1   & 0        & 5.01E-04 & 0        & 0        & 0        & 0        & 0        & 0        & 0        & 0        & 0        & 0        & 0        & 0         \\
2   & 5.01E-04 & 0        & 1.01E-03 & 4.51E-04 & 0        & 0        & 0        & 0        & 0        & 0        & 0        & 0        & 0        & 0         \\
3   & 0        & 1.01E-03 & 0        & 0        & 0        & 0        & 0        & 0        & 0        & 0        & 0        & 0        & 0        & 0         \\
4   & 0        & 4.51E-04 & 0        & 0        & 1.54E-04 & 0        & 0        & 0        & 0        & 0        & 0        & 0        & 0        & 0         \\
5   & 0        & 0        & 0        & 1.54E-04 & 0        & 8.75E-04 & 5.12E-04 & 0        & 0        & 0        & 0        & 0        & 0        & 0         \\
6   & 0        & 0        & 0        & 0        & 8.75E-04 & 0        & 0        & 0        & 0        & 0        & 0        & 0        & 0        & 0         \\
7   & 0        & 0        & 0        & 0        & 5.12E-04 & 0        & 0        & 7.23E-04 & 5.40E-04 & 3.56E-04 & 0        & 0        & 0        & 0         \\
8   & 0        & 0        & 0        & 0        & 0        & 0        & 7.23E-04 & 0        & 0        & 0        & 0        & 0        & 0        & 0         \\
9   & 0        & 0        & 0        & 0        & 0        & 0        & 5.40E-04 & 0        & 0        & 0        & 0        & 0        & 0        & 0         \\
10  & 0        & 0        & 0        & 0        & 0        & 0        & 3.56E-04 & 0        & 0        & 0        & 5.79E-04 & 0        & 0        & 0         \\
11  & 0        & 0        & 0        & 0        & 0        & 0        & 0        & 0        & 0        & 5.79E-04 & 0        & 1.67E-04 & 0        & 0         \\
12  & 0        & 0        & 0        & 0        & 0        & 0        & 0        & 0        & 0        & 0        & 1.67E-04 & 0        & 2.40E-04 & 2.46E-04  \\
13  & 0        & 0        & 0        & 0        & 0        & 0        & 0        & 0        & 0        & 0        & 0        & 2.40E-04 & 0        & 0         \\
14  & 0        & 0        & 0        & 0        & 0        & 0        & 0        & 0        & 0        & 0        & 0        & 2.46E-04 & 0        & 0         \\
\hline
\end{tabular}

\vspace{2\baselineskip}
\caption{ Reactance $\mathcal{X}_{ij}$}\label{table:reactance}
\setlength{\tabcolsep}{1pt}
\begin{tabular}{ccccccccccccccc} 
\hline
Bus & 1        & 2        & 3        & 4        & 5        & 6        & 7        & 8        & 9        & 10       & 11       & 12       & 13       & 14        \\ 
\hline
1   & 0        & 5.01E-04 & 0        & 0        & 0        & 0        & 0        & 0        & 0        & 0        & 0        & 0        & 0        & 0         \\
2   & 5.01E-04 & 0        & 1.01E-03 & 4.51E-04 & 0        & 0        & 0        & 0        & 0        & 0        & 0        & 0        & 0        & 0         \\
3   & 0        & 1.01E-03 & 0        & 0        & 0        & 0        & 0        & 0        & 0        & 0        & 0        & 0        & 0        & 0         \\
4   & 0        & 4.51E-04 & 0        & 0        & 1.54E-04 & 0        & 0        & 0        & 0        & 0        & 0        & 0        & 0        & 0         \\
5   & 0        & 0        & 0        & 1.54E-04 & 0        & 8.75E-04 & 5.12E-04 & 0        & 0        & 0        & 0        & 0        & 0        & 0         \\
6   & 0        & 0        & 0        & 0        & 8.75E-04 & 0        & 0        & 0        & 0        & 0        & 0        & 0        & 0        & 0         \\
7   & 0        & 0        & 0        & 0        & 5.12E-04 & 0        & 0        & 7.23E-04 & 5.40E-04 & 3.56E-04 & 0        & 0        & 0        & 0         \\
8   & 0        & 0        & 0        & 0        & 0        & 0        & 7.23E-04 & 0        & 0        & 0        & 0        & 0        & 0        & 0         \\
9   & 0        & 0        & 0        & 0        & 0        & 0        & 5.40E-04 & 0        & 0        & 0        & 0        & 0        & 0        & 0         \\
10  & 0        & 0        & 0        & 0        & 0        & 0        & 3.56E-04 & 0        & 0        & 0        & 5.79E-04 & 0        & 0        & 0         \\
11  & 0        & 0        & 0        & 0        & 0        & 0        & 0        & 0        & 0        & 5.79E-04 & 0        & 1.67E-04 & 0        & 0         \\
12  & 0        & 0        & 0        & 0        & 0        & 0        & 0        & 0        & 0        & 0        & 1.67E-04 & 0        & 2.40E-04 & 2.46E-04  \\
13  & 0        & 0        & 0        & 0        & 0        & 0        & 0        & 0        & 0        & 0        & 0        & 2.40E-04 & 0        & 0         \\
14  & 0        & 0        & 0        & 0        & 0        & 0        & 0        & 0        & 0        & 0        & 0        & 2.46E-04 & 0        & 0         \\
\hline
\end{tabular}
\end{sidewaystable}

\bibliographystyle{plain}
\bibliography{refs}

\end{document}